%% file: main.tex
\documentclass{siamart251216}

\usepackage{amssymb,mathrsfs}
\usepackage{bm}
\usepackage[T1]{fontenc}
\usepackage{microtype}
\usepackage{enumitem}
\usepackage{bookmark}
\usepackage{graphicx}
\usepackage[caption=false]{subfig}
\usepackage{placeins}
\usepackage{booktabs}
\usepackage{tikz-cd}
\usepackage{comment}
\usepackage{ulem}

\newsiamremark{remark}{Remark}
\newsiamremark{assumption}{Assumption}

\theoremstyle{plain}
\theoremheaderfont{\normalfont\itshape}
\theorembodyfont{\normalfont}
\theoremsymbol{}
\newtheorem*{unnumberedremark}{Remark}

\newcommand{\R}{\mathbb{R}}
\newcommand{\C}{\mathbb{C}}
\newcommand{\N}{\mathbb{N}}

\newcommand{\bbU}{\mathbb{U}}
\newcommand{\by}{\mathbf{y}}
\newcommand{\bz}{\mathbf{z}}

\newcommand{\bb}{\mathbf{b}}

\newcommand{\bk}{\mathbf{k}}

\newcommand{\balpha}{\boldsymbol{\alpha}}

\newcommand{\bnu}{\boldsymbol{\nu}}

\newcommand{\Om}{\Omega}

\newcommand{\Gcal}{\mathcal{G}}
\newcommand{\Ncal}{\mathcal{N}}

\newcommand{\abs}[1]{\left|#1\right|}
\newcommand{\norm}[1]{\left\|#1\right\|}

\newcommand{\ip}[2]{\left\langle #1,#2\right\rangle}

\newcommand{\dd}{\operatorname{d}}

\newcommand{\Phib}{\Phi^{\bb}}
\newcommand{\depth}{l}
\newcommand{\NN}{N}

\definecolor{darkred}{rgb}{.7,0,0}

\title{Operator Barron Regularity for Elliptic Equations and Application to Neural Operator Approximation\thanks{Submitted to the editors DATE.
\funding{
We acknowledge the funding support from 
National Key R\&D Program of China 2025YFA1018700, 
National Natural Science Foundation of China No.62595771,
National Natural Science Foundation of China No.12571436, 
the Fundamental Research Funds for the Central Universities,
NSF of the United States DMS-2309777, DMS-2309778, DMS-2542853.
}}}

\author{
  Shuhao Cao\thanks{
    School of Science and Engineering, University of Missouri-Kansas City
    (\email{scao@umkc.edu}).
  } 
  \and
  Ruchi Guo\thanks{
    Sichuan University
    (\email{ruchiguo@scu.edu.cn}).
  }
  \and
  Daniel Zhengyu Huang\thanks{
    Beijing International Center for Mathematical Research,  Center for Machine Learning Research, Peking University, Beijing, China 
    (\email{huangdz@bicmr.pku.edu.cn}).
  }
  \and
  Na Liu\thanks{School of Mathematical, Sichuan University, Chengdu, China
    (\email{
liuna0@stu.scu.edu.cn}).
  }
}

\headers
  {Barron Regularity and Operator Learning}
  {R. Guo, S. Cao, D.Z. Huang, and N. Liu}

\usepackage{etoolbox}
\makeatletter
\patchcmd{\normalsize}
  {\@setfontsize\normalsize\@xpt\@xiipt}
  {\@setfontsize\normalsize{9}{11}}
  {}{}
\makeatother
\normalsize

\apptocmd{\normalsize}{%
  \setlength{\abovedisplayskip}{5pt plus 1pt minus 1pt}
  \setlength{\belowdisplayskip}{5pt plus 1pt minus 1pt}
  \setlength{\abovedisplayshortskip}{0pt plus 1pt}
  \setlength{\belowdisplayshortskip}{5pt plus 1pt minus 1pt}
}{}{}
\normalsize

\begin{document}
\maketitle

\input{Chapters/0_0_Abstract}



\input{Chapters/1_Introduction}

\input{Chapters/2_Formulation}
\input{Chapters/3_harmonic}
\input{Chapters/3_Theory_coercive}
\input{Chapters/4_Error}
\input{Chapters/5_Experiment}

\input{Chapters/6_Summary}

\appendix
\input{Chapters/7_Appendix_v2}

% ...

\bibliographystyle{siamplain}

\end{document}

%% file: Chapters/0_0_Abstract.tex
\begin{abstract}
We develop a spectral Barron regularity framework for nonlinear operators between infinite-dimensional function spaces. 
Using Hilbert-valued Fourier representations of finite-dimensional input truncations, we establish bounds that are uniform in the input truncation dimension and prove arbitrary-order Barron regularity for a class of elliptic solution operators. 
Second, we use this operator Barron regularity to derive approximation error bounds for an FNO-type neural operator, with explicit dependence on the input truncation dimension, output reconstruction dimension, and network width. For a fixed input truncation dimension, optimizing the neural approximation bound over the Barron regularity order yields stretched-exponential decay of the approximation error in the network width. We also show some numerical experiments to support the predicted rate.
\end{abstract}


%% file: Chapters/1_Introduction.tex
\section{Introduction}

Partial differential equations (PDEs) with stochastic input data are a commonly used paradigm in science and engineering. Stochasticity typically reflects the uncertainty in the various parameters entering the physical phenomenon described by the equation. Neural operators provide surrogates for repeated PDE solves for different parameters without retraining, and their approximation properties are closely related to the regularity of the underlying solution maps. To present a new result regarding this relation, we consider the following abstract elliptic problem in a bounded Lipschitz domain \(\Omega\subset\mathbb R^d\):
\begin{equation}\label{eq:elliptic-pde}
\text{find }u(a)\in \mathbb V\text{ such that}\qquad
\mathfrak b(a;u(a),v)=\langle f,v\rangle_{\mathbb V',\mathbb V}
\quad\forall v\in \mathbb V.
\end{equation}
Here \(\mathbb V\) and \(\mathbb H\) be separable real Hilbert spaces, and \(\mathbb V\) is the solution space, incorporating the homogeneous essential boundary condition. The approximation error is measured in \(\mathbb H\). 
The coefficient field \(a\) is in an admissible subset of $L^\infty(\Omega)\cap \mathbb{R}^+$, \(f\in \mathbb V'\) is a fixed deterministic source term. \(\mathfrak b(a;\cdot,\cdot)\) is a continuous bilinear form on \(\mathbb V\times \mathbb V\). Assuming that \(\mathbb V\hookrightarrow \mathbb H\) is dense and compact, then there exists $C_S > 0$ depending on the spatial domain $\Omega$, 
\begin{equation}
\label{eq0_poincare}
\|v\|_{\mathbb H}\le C_S\|v\|_{\mathbb V},
\qquad v\in \mathbb V.
\end{equation}
By Lax-Milgram lemma, \eqref{eq:elliptic-pde} is well-posed, and the following are some specific examples:
\begin{itemize}
\item The diffusion equation with homogeneous Dirichlet boundary conditions:
\begin{equation}
\label{eq_case1_darcyflow}
\mathbb V=H_0^1(\Omega),\qquad \mathbb H=L^2(\Omega),\qquad
\mathfrak b(a;u,v)=\int_\Omega a(x)\nabla u(x)\cdot\nabla v(x)\,\dd x.
\end{equation}
In this case, \(C_S\) in \eqref{eq0_poincare} is the Poincar\'e constant when \(\|\cdot\|_{\mathbb V} = |\cdot|_{H^1}\).

\item The linear elasticity equation with homogeneous displacement boundary conditions:
\[
\begin{gathered}
\mathbb V=[H_0^1(\Omega)]^d,\qquad \mathbb H=[L^2(\Omega)]^d,\\[1mm]
\mathfrak b(a;u,v)=\int_\Omega a(x) \bigl( 2\mu \varepsilon(u)(x):\varepsilon(v)(x)+\lambda\operatorname{div}u(x)\operatorname{div}v(x) \bigl)\,\dd x.
\end{gathered}
\]
Here \(\mu>0\) and \(\lambda\ge0\) are fixed reference Lam\'e constants, \(\varepsilon(v)=\tfrac12(\nabla v+\nabla v^\top)\) is the linearized strain tensor, and \(:\) denotes the Frobenius inner product. In this case, \eqref{eq0_poincare} is the Korn's inequality when \(\|\cdot\|_{\mathbb V} =\|\varepsilon(\cdot)\|_{L^2(\Omega)}\). 
\end{itemize}
Here, for the purpose of a concise presentation, the coefficients in the examples above are determined by a single scalar coefficient field (see Section \ref{sec:barron-spaces}).
Nevertheless, the framework introduced in this paper can readily be extended to multiple or tensorial coefficient fields.

By the well-posedness of problem \eqref{eq:elliptic-pde}, the following coefficient-to-solution operator is $1$-$1$ and will be our main consideration throughout this paper:
\begin{equation}\label{eq:solution-operator}
\mathcal G^{\dagger}: \mathbb B^{\bb}\subset L^\infty(\Omega) \to \mathbb V,
\qquad g\mapsto u(\bar{a} + g),
\end{equation}
where $\bar a$ is a fixed background coefficient and
$\mathbb B^{\bb}$ is the prescribed family of coefficient
perturbations defined in \eqref{eq:barronfun}.
Our goal is to establish parametric spectral Barron regularity
of this operator under suitable assumptions on the coefficient
family and derive quantitative
approximation bounds in $\mathbb H$ for spectral-based neural operators.

\subsection{Literature review}
Neural operators have capacities to represent maps between function spaces.
Once trained, they can also be used as surrogates for repeated PDE solves in uncertainty quantification~\cite{guo2025warm,gao2024adaptive,cao2025derivative}, inverse problems \cite{2020GuoJiang,2023GuoCaoChenICLR,2026ZhengWeiCaoGuo,2025NelsenYang}, and engineering design~\cite{li2023fourier,wen2026geometry,huang2026cost,gomes2026combining}. 
Representative architectures include Deep Operator Networks (DeepONets)~\cite{lu2021deeponet,jin2022mionet}, whose branch and trunk networks represent input-dependent coefficients and spatial basis functions, respectively; 
kernel-based neural operators~\cite{2021NelsenStuart,li2020neural,li2021fourierneuraloperatorparametric,zeng2025point,han2026geometric}, including Fourier Neural Operators (FNOs) and their variants, which parameterize integral operators in Fourier space; and attention-based models~\cite{cao2021choose,2023LiMeidaniFarimani,wu2024transolver,2026TaiLiuLiChan}, whose continuum formulations use input-dependent integral kernels~\cite{calvello2025continuum}. 
Understanding their approximation properties requires regularity information about the underlying solution map. 
In this work, we investigate spectral Barron regularity to quantify the dependence of approximation error on input resolution, output resolution, and network width.

Barron-type spaces describe classes of functions admitting quantitative
approximation bounds by shallow neural networks.
Barron~\cite{Barron1993} established an $L^2$ approximation rate of
order $K^{-1/2}$ for width-$K$ sigmoidal networks.
Subsequent work developed integral representations~\cite{2022E}
and variation-space characterizations~\cite{2023SiegelXu} of
Barron-type spaces, and investigated functional-analytic properties
of spectral Barron spaces~\cite{choulli2025functional} and their
connections with classical smoothness spaces~\cite{meng2023new}.
Higher-order spectral Barron regularity also yields improved
approximation rates~\cite{SiegelXu2022HighOrder}.
For PDEs, spectral Barron regularity in the spatial variable
has been established for solutions of Poisson and static
Schr\"odinger equations~\cite{2021Lu,2023Chen,ming2025barron}. 
Barron approximations of variable-coefficient elliptic
PDEs have also been constructed; see \cite{2021Chen,2022EWojtowytsch} for example.
Barron norm can be also used in inverse problems as regularization \cite{lu2026interpolation}.
These mentioned existing works mainly focus on the spatial Barron regularity.

At the operator level, universal approximation results can
be obtained by combining finite-dimensional neural network
approximation with input encoding and output decoding~\cite{chen1995universal,2021Bhattacharya}.
For quantitative approximation bounds, one approach exploits PDE-specific structure through numerical solver emulation to construct neural operator approximations, 
which yields FNO error and complexity bounds for Darcy flow and Navier-Stokes equations~\cite{2021Kovachki}.
The solver emulation technique was also used in \cite{2023MarcatiSchwab} to achieve exponential convergence of DeepONet for elliptic equations.
A nonlinear reconstruction approach can also yield exponential-type convergence of shift-DeepONets and FNOs for a prescribed family of Burgers solutions with shocks~\cite{2023LanthalerMolinaroHadornMishra}.
Another important class of approaches relies on holomorphy-based analysis, originally established in \cite{2015CohenDeVore,Cohen2010Convergence,CohenDeVoreSchwab2011} by Cohen et al., for classical surrogate models;
see the more detailed discussion below.
We highlight that exploiting holomorphy of an extension to the parameter-to-solution map is also a key theoretical tool in the present work.

Two fundamental concepts heavily affect the analysis and the convergence rates in these works: the regularity of input functions and the regularity of the operator itself.
The aforementioned works employ different kinds of regularity of input functions in their setting. 
Roughly speaking, they can be categorized into two groups: the first one is based on classical continuity/differentiability including
Lipschitz continuity \cite{chen1995universal,weihs2026deep},
 $C^r$ spaces \cite{2026LanthalerStuart}, 
 analytic functions \cite{2023MarcatiSchwab}, 
 and Hilbert spaces \cite{2021Bhattacharya,2021Kovachki,schwab2026deep},
while the second one is based on certain summability/integrability of coefficients in expansions,
such as \cite{schwab2026deep,2015CohenDeVore,Cohen2010Convergence,CohenDeVoreSchwab2011}.
Recently, a lower bound was achieved in \cite{2026LanthalerStuart} to show that finite differentiability alone does not give a universal algebraic upper rate.


In contrast to regularity assumptions imposed on the input data, the regularity of a solution operator concerns the smoothness of its dependence on the input function and has been studied less systematically. For parametric elliptic PDEs, analytic or holomorphic dependence of the coefficient-to-solution map was established and developed in~\cite{Cohen2010Convergence,CohenDeVoreSchwab2011,2015CohenDeVore}. 
This structure has been exploited in neural-operator approximation: 
holomorphy underlies the elliptic DeepONet estimates in~\cite{2022Lanthaler}, 
and~\cite{Herrmann2024Neural} develops neural and spectral surrogates for general holomorphic maps between infinite-dimensional Hilbert spaces. Related neural-network approximation results use quantified holomorphy of high-dimensional parametric maps~\cite{2019SchwabZech,2022OpschoorSchwabZech}. 
Recently, this tool has also been used in the analysis of Kolmogorov widths \cite{2026LiZhang}.
Collectively, these works show that the holomorphic regularity of the solution and its solution operator provides a key mechanism for deriving dimension-robust approximation results.



Nevertheless, to the best of our knowledge, the Barron regularity of general operators has been far less explored in the literature than its function counterpart,
despite its significant role in the analysis of neural networks.
A variation norm-based Barron class is introduced for two-layer neural networks mapping between Banach spaces in \cite{2022Korolev}. 
For scalar- or finite-dimensional-valued functionals, \cite{2022YangXiang} defines a spectral Barron space through a Fourier-type representation on an infinite-dimensional input space. 
More recently, \cite{2026YangShi} formulates a Hilbert-valued variation space for nonlinear operators using vector-valued Borel measures; \cite{yan2026neural} shows the error estimates for Transformer-based neural operators in the form of scaling laws based on output functions' regularity and intrinsic dimensions.


\subsection{Contributions}

The contributions of the present work are threefold.
\begin{enumerate}
\item
We define a spectral Barron norm for nonlinear operators by combining finite-dimensional input truncations with Hilbert-valued Fourier representations.
The norm is independent of the output orthonormal basis and controls the square sum of the scalar Barron norms of the output coefficients in any such basis.

\item
Exploiting the ellipticity for problem~\eqref{eq:elliptic-pde} together with several analytical tools, we are able to establish an arbitrary-order Barron regularity for the associated coefficient-to-solution operator. 
In particular, with a minimal assumption on the uniform coercivity, we obtain a nearly optimal bound for the norm: $O(C^{\tau}(4+\tau)^\tau)$ with the Barron regularity order $\tau$, 
where hidden constants are only dependent on the prescribed coefficient family and the fixed problem data; see Theorem \ref{the:ridge-map-barron}.

\item
We show that, for operators with sufficient Barron regularity, FNO-type spectral neural operators yield a guaranteed approximation.
The neural approximation term admits an error bound in the order of $O(C^{\tau}(4+\tau)^\tau)
K^{-(\frac12+\frac{\tau}{N})}$, where $K$ stands for the network width. Additionally, the explicit dependence of the total error on the input truncation dimension $N$ and output dimension $M$ is identified; see Theorem \ref{thm:fno-approximation-error-with-complexity2}. Last but not least, we observe the predicted rates in the numerical experiments.
\end{enumerate}

To this end, the analysis unfolds with the following ingredients. 
First, truncations of input and output expansions reduce the operator of interest to a finite dimensional map, which weakens the proposition to become analyzing the uniformity with respect to truncations.
Second, complexification of the elliptic problems, quantitative holomorphic extension, and Cauchy estimates yield summable Taylor coefficients. 
We then employ delicate estimates in Harmonic Analysis to convert these coefficients into arbitrary-order spectral Barron bounds;
see detailed discussions in Section \ref{sec:harmonic}.
Third, the coefficient maps can then be represented by neural networks with higher order approximation rates. The network is realized as an FNO architecture, of which the encoding corresponds explicitly to the retained coefficients under a prescribed basis expansion, and the decoding naturally evaluates this expansion with the learned coefficient representations from the hidden layers that use cosine networks approximation.
Lastly, one can choose the Barron regularity order to decide the network width, which eventually yields the exponential rates.

In the next section, we discuss some basic setting of Barron spaces including development of operator Barron norms. 
In Section \ref{sec:mainResults}, we present the main problems to be addressed in this paper.
In Section \ref{sec:harmonic}, we prepare some fundamental tools in harmonic analysis.
In Sections \ref{sec:BarronNormResults} and \ref{sec:Error}, we rigorously estimate the Barron regularity and the approximation rate of an FNO-type neural operator.
Some basic lemmas together with the numerical results will be presented in the supplemental material.

\subsection{Notations}
\label{sec:notations}
As a broad range of notation is used in this paper, here we provide a brief overview of our notation conventions, and more specific meanings are given in context.
In general, there are several categories as follows. First, domain notations include $U$, $U_n, \bbU^n_r,\bbU_1$, with $\bbU$ usually denoting the complexified region. 
In addition, blackboard-bold letters are used for the principal spaces and sets, such as $\mathbb H$, $\mathbb V$, $\mathbb X$, and the Barron spaces including $\mathbb B_{\tau}$ and $\mathbb B^{\bb}$. $\mathbb X$'s complexification is denoted by $\mathbb X_{\mathbb C}=\{x_1+i x_2: \; x_1,x_2\in\mathbb X\}$.
Calligraphic letters are used for mappings, e.g., $\mathcal G$, $\mathcal N$, and network components such as $\mathcal P_N$, $\mathcal Q_p$, and $\mathcal L_\ell$.
Moreover, 
$\tau,s, k$ and $q$ are reserved for regularity or orders of approximations.
$d,n,m,N$ and $M$ denote dimensions, and specifically $N$ and $M$ indicate input and output truncation dimensions.
Bold lowercase letters denote vectors, sequences, or multi-indices; in particular, $\by_N$ and $\bz_N$ are real and complex parameter vectors. The subscript $e$ denotes an extension, as used in defining the Barron norm. 
Generic constants are denoted by $B$, $C$, and $D$, with any relevant dependence stated explicitly.

%% file: Chapters/2_Formulation.tex
\section{Barron regularity of functions and operators}
\label{sec:barron-spaces}

In this section, we recall spectral Barron spaces, introduce a
weighted coefficient parameterization of the input family, and
define a spectral Barron norm for operators.

Motivated by Barron's Fourier-moment condition~\cite{Barron1993},
we use the following definition of spectral Barron spaces, 
where $\tau$ denotes the regularity order throughout this paper.
\begin{definition}[The spectral Barron space]
\label{Def:SpatialNorm}
Let \(\tau>0\). The spectral Barron space on $U\subseteq \mathbb{R}^{n}$, $n \in \N$, is defined by
\begin{equation}
\label{eq:SpatialNorm}
\mathbb B_{\tau}(U)
:=
\Bigl\{
g:U\to\mathbb C:
\|g\|_{\mathbb B_{\tau}(U)}
:=
\inf_{ g_e|_{U}=g}
\int_{\mathbb R^{n}}
\left(1+\|\xi\|_1^\tau \right) |\widehat{g}_e(\xi)|\,\dd\xi
<\infty
\Bigr\}.\end{equation}
Here, the infimum is taken over integrable extensions $g_e$ to $\mathbb R^{n}$ whose Fourier transforms
$\widehat {g}_e(\xi) = \int_{\R^{n}} g_e(x)e^{-2\pi i\ip{x}{\xi}}\,\dd x$
satisfy the displayed weighted integrability condition. The restriction
$g_e|_U=g$ is understood for the continuous representative given by Fourier inversion.
\end{definition}

In the present work, 
we consider real-valued parametric input functions on $\Omega$ through a weighted sequence space.
For this purpose, given
a sequence of weights 
$\{w_j\}_{j\in\mathbb N_0}$ satisfying $w_j\ge1$, we introduce
\begin{equation}
\label{eq_l1}
\ell_{w} := \left\{ \balpha=(\alpha_j)_{j\in\N_0}\in\mathbb R^{\mathbb N_0}: \sum\nolimits_{j \in \N_0} w_j \abs{\alpha_j} < \infty \right\},
~~~~
\| \balpha \|_{\ell_w}
:=
\sum_{j \in \N_0} w_j |\alpha_j|.
\end{equation}
Following~\cite{2015CohenDeVore,Cohen2010Convergence,CohenDeVoreSchwab2011},
let $\psi_j\in L^\infty(\Omega;\mathbb R)$ satisfy
$\|\psi_j\|_{L^\infty(\Omega)}\le1$ for $j\in\mathbb N_0$,
and denote $U_\infty:=[-1,1]^{\mathbb N_0}$ containing all $\mathbf{y} = (y_j)_{j\in\mathbb N_0}$. 
For a fixed $\bb=(b_j)_{j\in\mathbb N_0}\in\ell_w$, define the input family
\begin{equation}
\label{eq:barronfun}
\mathbb{B}^{\bb}  (\Omega) := \left\{ 
g(x;\by) = \sum\nolimits_{j \in \N_0}b_j y_j\psi_j(x):
~~
 \by \in U_\infty 
 \right\}.
\end{equation}
Since $\ell_w\subseteq\ell^1$, the series converges absolutely
in $L^\infty(\Omega)$ and defines the parameterization
\begin{equation}
\label{eq_Imap}
\Phib: U_\infty \rightarrow \mathbb{B}^{\bb}(\Omega) \subset L^\infty(\Omega;\mathbb R),
\qquad
\by \mapsto \left( \sum_{ j \in \N_0} b_j y_j \psi_j (x)\right)|_{\Omega}.
\end{equation}

\begin{remark}
The input family consists of spatial Barron functions for a
suitable choice of modes and weights. Fix $p>0$ and choose a box
$Q_L:=(-L/2,L/2)^d$ containing $\overline\Omega$.
Consider the real trigonometric basis
\begin{equation}
\label{eq_basis}
\psi_{0}(x) := 1,
~~~
\psi_{2j-1}(x) := \cos\left( \frac{2\pi\ip{\bk(2j-1)}{x}}{L}\right),
~~~
\psi_{2j}(x) := \sin\left( \frac{2\pi\ip{\bk(2j)}{x}}{L}\right),
~~ j\ge 1,
\end{equation}
where $\{\bk(j)\}_{j\in\mathbb N_0}$ is a fixed enumeration of
$\mathbb Z^d$ such that $\bk(0)=\mathbf 0$ and
$\bk(2j-1)=-\bk(2j)$ for $j\geq 1$,
and they satisfy $\| \psi_j \|_{L^{\infty}\left( Q_L \right)} = 1$. 
For these modes, set $w_j:=1+\|\bk(j)/L\|_1^p$, for the prescribed $\bb\in\ell_w$, the parametrized family $\mathbb B^\bb(\Omega)$ is uniformly bounded in $\mathbb B_p(\Omega)$ as the image of a coefficient box, since $\sup_{\by\in U_\infty}
\|\Phib(\by)\|_{\mathbb B_p(\Omega)}
\le C_{\Omega,p}\|\bb\|_{\ell_w}$ (see Lemma~\ref{lem:lem_phi1}).
The general parameterization~\eqref{eq_Imap} also allows other
uniformly bounded real modes, including suitably normalized
Karhunen--Lo\`eve modes or wavelets, provided that the amplitudes
satisfy the stated summability condition. Such representations
are used in both neural operator design and the error analysis; see, e.g.,
\cite{2021Bhattacharya,lu2021deeponet,li2021fourierneuraloperatorparametric,gupta2021multiwavelet}.

\end{remark}

To define a spectral Barron norm for operators, we first extend
the scalar definition \eqref{eq:SpatialNorm} to Hilbert-valued functions, using the Fourier transform on Bochner spaces (e.g.,  see
\cite[Section~2.4]{Hytonen2016Analysis}).
For $N\ge1$, consider the domain $U_N:=[-1,1]^N$.
\begin{definition}[Hilbert-valued spectral Barron norms]
\label{Def:HilbertBarron}
Let $\mathbb X$ be a separable real Hilbert space and let $\tau>0$.
For $G:U_N\to \mathbb X$, define
\begin{equation}
\label{eq:hilbert-barron-norm}
\|G\|_{\mathbb B_\tau(U_N;\mathbb X)}
:=
\inf_{G_e|_{U_N}=G}
\int_{\mathbb R^N}
(1+\|\xi\|_1^\tau)
\|\widehat G_e(\xi)\|_{\mathbb X_{\mathbb C}}\,d\xi.
\end{equation}
Here, $\mathbb X_{\mathbb C}$ is the complexification of $\mathbb X$, and the
infimum is over integrable extensions $G_e:\mathbb R^N\to \mathbb X_{\mathbb C}$
for which the integral above is finite. The Fourier transform is taken in
the parameter variables using Bochner integration, and restrictions
are understood through Fourier inversion, as in
Definition~\ref{Def:SpatialNorm}. If no such extension exists,
the norm is $+\infty$.
\end{definition}

Finally, we extend Definition~\ref{Def:HilbertBarron} from $U_N$ to the input family $\mathbb{B}^{\bb}$~\eqref{eq:barronfun} using the truncated parameterization
\begin{equation}\label{eq:truncation}
  {
\Phib \circ\mathcal P_N^\top:
U_N\to\mathbb B^\bb(\Omega),
\qquad
\by_N\mapsto g_N(\cdot;\by_N) := \sum_{j=0}^{N-1}b_jy_j\psi_j.
    }
\end{equation}
Here $\mathcal P_N:U_\infty \to U_N$ retains the first $N$ parameters, for $\by\in U_\infty$,
$\mathcal P_N \by=\by_N:=(y_0,\ldots,y_{N-1})\in U_N$,
and $\mathcal P_N^\top:U_N\to U_\infty$ appends zeros.
We then define the operator spectral Barron norm.

\begin{definition}[Operator spectral Barron norms]
\label{def:operator-barron}
Fix a separable real Hilbert space $\mathbb X$ and $\tau>0$.
For $\mathcal M:\mathbb B^{\bb}(\Omega)\to \mathbb X$, 
continuous with respect to $\|\cdot\|_{L^{\infty}(\Omega)}$- and $\|\cdot\|_{\mathbb X}$-norm topologies, respectively,
define its
input-truncated parameterization by
\begin{equation}
\label{eq_operator_decomp}
\mathcal{M}_N = \mathcal{M} \circ \Phib \circ \mathcal P^{\top}_N : 
~
U_N \rightarrow \mathbb X.
\end{equation}
The operator spectral Barron norm is
\begin{equation}
\label{eq_operatornorm}
\|\mathcal M\|_{\mathbb B_\tau(\mathbb B^{\bb};\mathbb X)}
:=
\sup_{N\geq1}
\|\mathcal M_N\|_{\mathbb B_\tau(U_N;\mathbb X)}.
\end{equation}
\end{definition}

\begin{lemma}[Barron bounds by basis]
\label{lem:vector-barron-embedding}
Let $\{e_m\}_{m\in \Lambda}$ be an orthonormal basis of $\mathbb X$, and let
$\mathcal M:\mathbb B^{\bb}(\Omega)\to \mathbb X$ be a continuous
operator with finite operator Barron norm defined in \eqref{eq_operatornorm}.
For every $N\ge1$, its input-truncated parameterization satisfies
\begin{equation}
\label{eq:barron-coordinate-bound}
\left(\sum_{m\in \Lambda}
\|\langle \mathcal M_N, e_m\rangle_{\mathbb X}\|_{\mathbb B_\tau(U_N)}^2
\right)^{1/2}
\le\|\mathcal M_N\|_{\mathbb B_\tau(U_N;\mathbb X)} \leq \|\mathcal M\|_{\mathbb B_\tau(\mathbb B^{\bb};\mathbb X)}.
\end{equation}
The sum is understood as finite when $\mathbb X$ is finite dimensional.
\end{lemma}
\begin{proof}
For any admissible extension $\mathcal M_{N,e}$, the scalar
function $\langle\mathcal M_{N,e},e_m\rangle_{\mathbb X_{\mathbb C}}$
extends $\langle\mathcal M_N,e_m\rangle_{\mathbb X}$.
The definition of the scalar norm, Minkowski's inequality,
and Parseval's identity in $\mathbb X_{\mathbb C}$ give
\[
\begin{aligned}
\left(\sum_m\|\langle \mathcal M_{N},e_m\rangle_{\mathbb X}\|_{\mathbb B_\tau(U_N)}^2\right)^{1/2}
&\le\left(\sum_m
\left[\int_{\mathbb R^N}(1+\|\xi\|_1^\tau)
|\langle\widehat{\mathcal M}_{N,e}(\xi),e_m\rangle_{\mathbb X_{\mathbb C}}|\,d\xi\right]^2
\right)^{1/2}\\
&\le\int_{\mathbb R^N}(1+\|\xi\|_1^\tau)
\|\widehat{\mathcal M}_{N,e}(\xi)\|_{\mathbb X_{\mathbb C}}\,d\xi.
\end{aligned}
\]
Taking the infimum over $\mathcal M_{N,e}$ proves the first
inequality in~\eqref{eq:barron-coordinate-bound}.
The second follows from~\eqref{eq_operatornorm}.
\end{proof}

\begin{remark}[Output basis and neural reconstruction]
\label{rem:barron-output-reconstruction}
Although the norm in \eqref{eq_operatornorm} is independent of the output basis of $\mathbb{X}$, since \eqref{eq:hilbert-barron-norm} uses only
the Hilbert norm, it does depend on the prescribed input
parameterization $\Phib$ and its truncations.
For neural approximation, one may still choose an orthonormal
basis $\{e_m\}$ and reconstruct the output as
$\sum_{m=1}^M N_{m,K}(\by_N)e_m$, where
$N_{m,K}$ approximates
$F_{m,N}=\langle\mathcal M_N,e_m\rangle_{\mathbb X}$.
By \eqref{eq:barron-coordinate-bound}, 
the sum of the squared
Barron norms of these coefficient maps is bounded by
$\|\mathcal M_N\|_{\mathbb B_\tau(U_N;\mathbb X)}^2$,
uniformly in $M$. 
Consequently, the operator Barron norm yields neural
approximation bounds for the retained coefficients that are
uniform over all orthonormal output bases.
The output truncation error from discarding the remaining coefficients depends
on the chosen basis; see Remark \ref{rem_basis} for details.
\end{remark}




\section{The target approximation problem and the main result}
\label{sec:mainResults}
In this section, we formulate the operator approximation problem
for a parametric linear elliptic equation 
\[
\mathfrak b(a;u(a),v)
=
\langle f,v\rangle_{\mathbb V',\mathbb V}
\qquad\text{for all }v\in \mathbb V,
\] and establish
Barron regularity and approximation error estimates for the
associated solution operator. 
We consider coefficients of the form
\[
a(x)=\bar{a}(x)+g(x),
~~~~
\bar{a}\in L^\infty(\Omega;\mathbb R),
~~~
g\in\mathbb B^{\bb}(\Omega).
\]
Using the parameterization in \eqref{eq:barronfun}, the coefficient and its
truncation to the first $N$ parameters can be written as
\begin{equation}
\label{eq:target-coefficient-expansion}
a(x;\by)
:=
\bar{a}(x)
+
\sum_{j\in\mathbb N_0}b_jy_j\psi_j(x),
\qquad
a_{N}(x;\by_N) = \bar{a}(x) + \sum_{j\le N-1} b_jy_j \psi_j(x),
\end{equation}
where $\by=(y_j)_{j\in\mathbb N_0}\in U_\infty$ and $\by_N=(y_0,\ldots,y_{N-1})\in U_N=[-1,1]^N$.

For the operator approximation, we consider the corresponding truncated variational problem:
find $u_N(\by_N)\in \mathbb V$ such that
\begin{equation}\label{eq:elliptic-pde_trunc}
\begin{gathered}
\mathfrak b\bigl(a_N(\cdot,\by_N);u_N(\by_N),v\bigr)
=\langle f,v\rangle_{\mathbb V',\mathbb V},
\qquad \forall v\in \mathbb V.
\end{gathered}
\end{equation}
When no ambiguity arises, we write $a_N=a_N(\cdot,\by_N)$ and $u_N=u_N(\by_N)$ for simplicity.
We impose the following assumptions to ensure uniform
well-posedness of the variational problems and to establish
Barron regularity of the solution operator.

\begin{assumption}\label{ass:ellipticity}
Let $\mathbb V$ be the real Hilbert space introduced above, and let $f\in \mathbb V'$ be fixed.
Let $\bar{a}\in L^\infty(\Omega;\mathbb R)$, let $\bb\in\ell_w$ with $w_j\ge1$, and let the real coefficient family be given by \eqref{eq:target-coefficient-expansion}.
There exist constants $C_{\mathfrak b}>0$, $C^{\star}>0$, and $C_{\star}>0$, independent of the parameter values and truncation levels, 
such that for all $u,v\in \mathbb V$,
\begin{enumerate}[label={(A\arabic*)},ref={A\arabic*},leftmargin=*]
\item\label{ass:ellipticity-A1} 
For every $a\in L^\infty(\Omega;\mathbb R)$, $\mathfrak b(a;\cdot,\cdot)$ is a bounded, symmetric, real bilinear form on $\mathbb V\times \mathbb V$.
Moreover, there holds
\[
|\mathfrak b(a;u,v)-\mathfrak b(\widetilde a;u,v)|
\le C_{\mathfrak b}\|a-\widetilde a\|_{L^\infty(\Omega)}\|u\|_{\mathbb V}\|v\|_{\mathbb V},
~~~
\forall a,\widetilde a\in L^\infty(\Omega;\mathbb R).
\]
\item\label{ass:ellipticity-A2}
The forms are uniformly bounded and coercive on the real coefficient family:
\[
|\mathfrak b(a(\cdot;\by);u,v)|
\le C^{\star}\|u\|_{\mathbb V}\|v\|_{\mathbb V},
\qquad
\mathfrak b(a(\cdot;\by);u,u)
\ge C_{\star}\|u\|_{\mathbb V}^2,
~~~
\forall \by\in U_\infty.
\]
The same bounds apply to $a_N(\cdot;\by_N)$ by extending $\by_N$ with zeros.
\item\label{ass:ellipticity-B2} 
Define $\mathbb V_{\mathbb C} = \mathbb V + i\mathbb V$.
The form $\mathfrak{b}$ extends to $\mathbb V_{\mathbb C}\times \mathbb V_{\mathbb C}$, linearly in the first argument and conjugate-linearly in the second, and the coefficient-difference estimate in (\ref{ass:ellipticity-A1}) holds for complex coefficients and complex arguments.
$a\mapsto\mathfrak b(a;\cdot,\cdot)$ is complex-affine, i.e.,
\[
\mathfrak b((1-t)a+t\widetilde a;u,v)
=(1-t)\mathfrak b(a;u,v)
+t\mathfrak b(\widetilde a;u,v),
~~~ \forall t\in \mathbb{C}.
\]
\end{enumerate}
\end{assumption}

\begin{unnumberedremark}
Assumption~(\ref{ass:ellipticity-A2}) directly imposes uniform boundedness and coercivity on the admissible family; no smallness condition on $\|\bb\|_{\ell_w}$ is imposed.
The form acts on the entire coefficient field $a(\cdot;\by)$, so these inequalities are quantified over $\by$ and $u,v$, rather than pointwise in $x$.
For the scalar diffusion problem, when $\|v\|_{\mathbb V}=\|\nabla v\|_{L^2}$, 
the pointwise bounds
$0<C_{\star}\le a(x;\by)\le C^{\star}$ for almost every $x\in\Omega$ and every admissible $\by$ imply this condition.
Assumptions~(\ref{ass:ellipticity-A1})-(\ref{ass:ellipticity-B2}) yield holomorphic extensions to complex neighborhoods of the real parameter cubes.
The summability condition $\bb\in\ell_w\subset\ell^1$
allows us to split $\bb$ into a fixed finite set
and a sufficiently small tail, as required by
\eqref{eq:ridge-tail-smallness} in the Barron regularity proof.
\end{unnumberedremark}

\begin{theorem}
\label{thm:real-well-posedness}
Under assumptions~(\ref{ass:ellipticity-A1})--(\ref{ass:ellipticity-A2}), the variational problems \eqref{eq:elliptic-pde} and \eqref{eq:elliptic-pde_trunc} admit unique solutions for every admissible real parameter, with
\[
\|u(\by)\|_{\mathbb V}\le C_{\star}^{-1}\|f\|_{\mathbb V'},
\qquad
\|u_N(\by_N)\|_{\mathbb V}\le C_{\star}^{-1}\|f\|_{\mathbb V'},
\qquad N\ge1.
\]
\end{theorem}
\begin{proof}
It simply follows from the real Lax--Milgram theorem and the uniform bounds in Assumption~(\ref{ass:ellipticity-A2}). 
\end{proof}

Under Assumption~\ref{ass:ellipticity}, for each admissible real
parameter $\by\in U_\infty$ and each truncated parameter
$\by_N\in U_N$, let $u(\by)\in \mathbb V$ and
$u_N(\by_N)\in \mathbb V$ denote the solutions corresponding to
$a(\cdot;\by)$ and $a_N(\cdot;\by_N)$, respectively.
The preceding theorem gives uniform well-posedness on the prescribed coefficient family $\mathbb B^{\bb}(\Omega)$.

Define the centered coefficient-to-solution operator by
\begin{equation}
\label{solution_opl_eq0}
\mathcal G^\dagger:\mathbb B^{\bb}(\Omega) \to \mathbb V,
\qquad
g\mapsto u(\bar{a}+g).
\end{equation}
Equivalently,
$\mathcal G^\dagger(a-\bar{a})=u(a)$. Centering the input at $\bar{a}$
aligns the operator domain with the parameterization of the
coefficient perturbations.

To study the Barron regularity of $\mathcal G^\dagger$ within the
framework of \eqref{eq_operatornorm}, define the induced
truncated parameter-to-solution maps
\begin{equation}
\label{eq_Gdecomp}
\begin{aligned}
\mathcal G^{\dagger}_N
&:=
\mathcal G^\dagger
\circ\Phib
\circ\mathcal P^{\top}_N
:
U_N\to \mathbb V.
\end{aligned}
\end{equation}
Thus,  $\mathcal G^{\dagger}_N(\by_N)=u_N(\by_N)$.
For output, fix an $\mathbb H$-orthonormal basis
$\{e_m\}_{m\ge1}\subset \mathbb V$, and define 
\begin{equation}
\label{eq_VM_proj}
\mathbb V_M:=\operatorname{span}\{e_1,\ldots,e_M\}\subset \mathbb V,
~~
\text{and}
~~
\text{$\mathbb H$-orthogonal projection} ~ \Pi_M:\mathbb H\to \mathbb V_M
\end{equation}
Write
\begin{equation}
\label{eq:solution-spectral-expansion}
u_N(\by_N)
=
\sum_{m=1}^{\infty}F_{m,N}(\by_N)e_m,
~~~~~
F_{m,N}(\by_N):= \langle u_N(\by_N),e_m\rangle_{\mathbb H},
\end{equation}
with convergence in $\mathbb H$. The operator Barron norm is
\begin{equation}
\label{cts_barronnorm}
\|\mathcal G^\dagger\|_{\mathbb B_\tau(\mathbb B^{\bb};\mathbb H)}
=
\sup_{N\geq1}
\|u_N\|_{\mathbb B_\tau(U_N;\mathbb H)}.
\end{equation}
By Lemma~\ref{lem:vector-barron-embedding}, this norm controls the
square sum of the scalar Barron norms of $F_{m,N}$ for any chosen
output basis. The following theorem gives a bound uniform in the
input truncation level. Its proof is given in
Section~\ref{ssec:proof-barron}.

\begin{theorem}
\label{the:ridge-map-barron}
Under Assumption~\ref{ass:ellipticity}, 
the solution operator satisfies
$\mathcal G^\dagger\in C(\mathbb B^{\bb}(\Omega),\mathbb H)$,
where the input family carries the $L^\infty(\Omega)$ topology
and $\mathbb H$ carries its norm topology.
Moreover, there exist constants
$C_B,D_{B}\ge1$, depending on the fixed problem data and the prescribed data $\bb$, 
but independent of $\tau$, $M$, and $N$, such that, for every $\tau>0$,
\begin{equation}
\label{eq:ridge-map-Gamma}
\begin{aligned}
\|\mathcal G^\dagger\|_{\mathbb B_\tau(\mathbb B^{\bb};\mathbb H)}
&\le C_S
\|\mathcal G^\dagger\|_{\mathbb B_\tau(\mathbb B^{\bb};\mathbb V)}   \le C_B [D_{ B}(4+\tau)]^\tau.
\end{aligned}
\end{equation}
Both operator norms are those of Definition~\ref{def:operator-barron}.
\end{theorem}


\begin{remark}
The dependence on $\tau$ in \eqref{eq:ridge-map-Gamma}
is almost optimal.
Fix a nonzero forcing $f\in\mathbb V'$ and consider scalar
diffusion with homogeneous Dirichlet boundary conditions
and the spatially constant coefficient
$a(x;y)=2-y$, $y\in[-1,1]$.
This corresponds to
$\bar a=2$, $b_0=1$, $\psi_0=-1$, and $b_j=0$ for $j\ge1$.
Let $u_f\in\mathbb V$ denote the solution with unit coefficient
and the same forcing $f$. 
Then $u(y)=u_f/(2-y)$.
For $\mathbb X\in\{\mathbb V,\mathbb H\}$, every integer $m\ge1$, and every admissible
extension $u_e$ in \eqref{eq:hilbert-barron-norm} with $\tau=m$,
differentiation with respect to $y$ under the Fourier integral gives
\[
\frac{m!}{2^{m+1}}\|u_f\|_{\mathbb X}
=\|u^{(m)}(0)\|_{\mathbb X}
\le(2\pi)^m\int_{\R}|\xi|^m
\|\widehat u_e(\xi)\|_{\mathbb X_{\C}}\,\dd\xi.
\]
Taking the infimum over these extensions and using
\eqref{eq_operatornorm}, we obtain
\[
\begin{aligned}
\|\mathcal G^\dagger\|_{\mathbb B_m(\mathbb B^{\bb};\mathbb X)}
\ge \|\mathcal G_1^\dagger\|_{\mathbb B_m([-1,1];\mathbb X)}
 = \|u\|_{\mathbb B_m([-1,1];\mathbb X)} 
\ge\frac{m!}{2(4\pi)^m}\|u_f\|_{\mathbb X}
\ge\frac12\left(\frac{m}{4\pi e}\right)^m\|u_f\|_{\mathbb X}.
\end{aligned}
\]
Here the last inequality follows from $m!\ge(m/e)^m$.
Thus, for $\tau=m$, the $m^m$ growth in
\eqref{eq:ridge-map-Gamma} is sharp up to factors exponential in $m$.
\end{remark}

Next, for FNO approximation, we specify the measure.
For each \(j\in\mathbb N_0\), let $\rho^{(j)}$ be the uniform probability measure on $[-1,1]$, with $\dd\rho^{(j)}(t)=\frac12\,\dd t$.
Equip $U_\infty$ with its product $\sigma$-algebra and product probability measure
$\rho:=\bigotimes_{j\in\mathbb N_0}\rho^{(j)}$, and let
\begin{equation}
\label{eq_prob_measure}
\rho_N:=\bigotimes_{j<N}\rho^{(j)},
\qquad
\dd\rho_N(\by_N)=2^{-N}\prod_{j<N}\dd y_j.
\end{equation}
For a Hilbert space $\mathbb X$, we write
\[
\|v_N\|_{L^2(U_N,\rho_N;\mathbb X)}^2
:=\int_{U_N}\|v_N(\by_N)\|_{\mathbb X}^2\,\dd\rho_N(\by_N).
\]
For scalar-valued functions, the target space $\mathbb X=\mathbb R$ is omitted.
All parameter $L^2$ norms below use these probability measures;
the scalar and Hilbert-valued spectral Barron norms are given by
Definitions~\ref{Def:SpatialNorm} and~\ref{Def:HilbertBarron}, respectively.
The $L^\infty(\Omega)$ stability estimate \eqref{eq:solution-lipschitz}, together with the uniform solution bound, ensures that the solution maps considered below are measurable and belong to the corresponding parameter $L^2$ spaces.

Now, we present the following two main theorems of this work whose detailed proofs are postponed to Sections \ref{subsec:thm_error1} and \ref{subsec:error_2}.

\begin{theorem}
\label{thm:fno-approximation-error-with-complexity}
Under Assumption~\ref{ass:ellipticity}, 
let $N,M,K\ge1$ be integers,
and let $\Omega$ be a periodic domain.
Given the trigonometric parameterization of the input data and an $\mathbb H$-orthonormal basis $\{e_m\}_{m=1}^{\infty}\subset \mathbb V$,
{for every $\tau >0$}, there exists an FNO-type spectral neural operator described in Section~\ref{sec:Error},
\(
\mathcal{N}: \mathbb{B}^{\bb} \to \mathbb V
\)
with width $\max\{4 N, MK, 4M\}$,
such that
\begin{equation}
\label{eq:fno-total-error-with-complexity-high}
\begin{aligned}
&\sup_{n \ge N} \norm{ \mathcal{G}^{\dag}(g_n)-\Ncal(g_n) }_{{L^2(U_n,\rho_n; \mathbb H)}} \\
&\le
\varepsilon_N^{\mathrm{in}}
+ \varepsilon_{N,M}^{\mathrm{out}}
+
4\left( \max\{8\pi e,2(\tau+1)\}\right)^\tau \| \mathcal{G}^{\dag} \|_{\mathbb B_\tau({\mathbb B^{\bb};\mathbb H})} K^{-(\frac12+\frac{\tau}{N})},
\end{aligned}
\end{equation}
where 
\begin{equation}
\label{eq:fno-total-error-with-complexity-high-1}
\begin{aligned}
\varepsilon_N^{\mathrm{in}}
:=\sup_{n\ge N}
\norm{\mathcal G^\dagger(g_n)-\mathcal G^\dagger(g_N)}
_{L^2(U_n,\rho_n;\mathbb H)},
~~~~
\varepsilon_{N,M}^{\mathrm{out}}
:=
\norm{(I-\Pi_M)u_N}_{L^2(U_N,\rho_N;\mathbb H)}.
\end{aligned}
\end{equation}
Here $\varepsilon^{\mathrm{in}}_N$ measures the input truncation error, tending to zero as $N\to\infty$,
and $\varepsilon_{N,M}^{\mathrm{out}}$ measures the reconstruction error of $u_N$, tending to zero as $M\to\infty$ for each $N$, 
with $\Pi_M$ in \eqref{eq_VM_proj}.
\end{theorem}

Combining the $\mathbb H$-valued operator norm estimate in
Theorem~\ref{the:ridge-map-barron} with the explicit approximation
constant in Lemma~\ref{lem:explicit-BarronApprox}, we choose the
Barron regularity index $\tau$ in terms of the network width $K$
to obtain a nearly exponential decay of the neural approximation error.

\begin{theorem}
\label{thm:fno-approximation-error-with-complexity2}
Under the same assumption as Theorem~\ref{thm:fno-approximation-error-with-complexity}, fix the constant $D_{B}$ in Theorem~\ref{the:ridge-map-barron}, let $N,M,K\ge1$ be integers, for a fixed input truncation $N$, we choose the regularity order $\tau$ in \eqref{eq:fno-total-error-with-complexity-high} for $\mathcal{G}^{\dag}$ to be
\begin{equation}
\label{eq:fno-optimal-tau} 
\tau = \tau_K:=\frac{K^{1/(2N)}}{e\sqrt{4\pi e D_{B} }}.
\end{equation}
Then there exists a network of the form specified in Theorem~\ref{thm:fno-approximation-error-with-complexity} such that
\begin{equation}
\label{eq:fno-optimized-error}
\begin{aligned}
{\sup_{n \ge N} \norm{ \mathcal{G}^{\dag}(g_n)-\Ncal(g_n) }_{{L^2(U_n,\rho_n; \mathbb H)}} }
\le \;
\varepsilon_N^{\mathrm{in}}
+ 
\varepsilon_{N,M}^{\mathrm{out}}
+
C_{\mathcal{N}} K^{-1/2}
\exp (-2\tau_K).
\end{aligned}
\end{equation}
Here $\varepsilon_N^{\mathrm{in}}$, $\varepsilon_{N,M}^{\mathrm{out}}$, and $D_B$ are inherited
from Theorems~\ref{thm:fno-approximation-error-with-complexity} and \ref{the:ridge-map-barron},
and $C_{\mathcal{N}}$ is a constant independent of $N$, $M$, $K$ and $\tau_K$.
\end{theorem}

\begin{remark}
\label{rem_basis}
For the output reconstruction in
\eqref{eq:solution-spectral-expansion}, we can use the generalized eigenfunctions of
the background coefficient. Since $\mathbb V\hookrightarrow \mathbb H$ is dense and
compact and $\mathfrak b(\bar{a};\cdot,\cdot)$ is continuous,
symmetric, and coercive, the generalized eigenvalue problem
\begin{equation}
\label{eq:background-eigenproblem}
\mathfrak b(\bar{a};e_m,v)
=
\lambda_m\langle e_m,v\rangle_{\mathbb H},
\qquad v\in \mathbb V,
\end{equation}
admits an $\mathbb H$-orthonormal basis
$\{e_m\}_{m=1}^{\infty}\subset \mathbb V$ satisfying
$\langle e_m,e_n\rangle_{\mathbb H}=\delta_{mn},
0<\lambda_1\leq\lambda_2\leq\cdots,
\lambda_m\to\infty$;
see, e.g.,~\cite{2011Brezis}.
For this basis, the identity $\mathfrak b(\bar{a};v,v)=\sum_{m=1}^{\infty}\lambda_m\abs{\langle v,e_m\rangle_{\mathbb H}}^2$ for $v\in \mathbb V$ and Assumption~\ref{ass:ellipticity-A2} clearly yield
$\varepsilon_{N,M}^{\mathrm{out}}
\le C^f\sqrt{C^{\star}}\,\lambda_{M+1}^{-1/2}$.
Assume in addition that the background eigenvalues satisfy the second-order Weyl asymptotic law
$\lambda_m\sim c_{\mathrm W}m^{2/d}$, $m\to\infty$, with $c_{\mathrm W}>0$;
see, e.g.,~\cite[Section~6.5]{evans2010partial} for the Dirichlet Laplacian.
Then there exists $c_\lambda>0$ such that $\lambda_m\ge c_\lambda m^{2/d}$ for all $m\ge1$, and consequently
\[
\varepsilon_{N,M}^{\mathrm{out}}
\le C^f\sqrt{\frac{C^{\star}}{c_\lambda}}\,(M+1)^{-1/d},
\qquad
\varepsilon_{N,M}^{\mathrm{out}} = O(M^{-1/d}).
\]
When the solution is smoother, the approximation rate with respect to $M$ is also higher.
\end{remark}



%% file: Chapters/3_harmonic.tex

\section{Basic results in harmonic analysis}
\label{sec:harmonic}

In this section, we establish Fourier moment estimates for
extensions of Hilbert-valued functions, with explicit dependence
on the moment order and dimension. These estimates will be used
to bound the spectral Barron norm of the solution operator
in the next section.

Throughout this section, we consider $\mathbb X_{\C}$ for a separable real Hilbert
space $\mathbb X$.
Let $F_e\in L^1(\R^n;\mathbb X_{\C})$ be an extension
of $F:[-1,1]^n\to \mathbb X$. 
We define
\[
M_s(F_e)
:=
\int_{\R^n}
\|\xi\|_1^s\|\widehat F_e(\xi)\|_{\mathbb X_{\C}}\,\dd\xi,
\qquad
M_0(F_e)
:=
\int_{\R^n}
\|\widehat F_e(\xi)\|_{\mathbb X_{\C}}\,\dd\xi,
~~~
s>0.
\]
The quantities $M_s$ are widely used in harmonic analysis; see, for example \cite[Sections~1.4 and~2.3]{Lukacs1970}.
For scalar-valued extensions, the norm $\|\, \cdot \,\|_{\mathbb X_{\C}}$ of the Fourier transform is
replaced by its absolute value. By
Definition~\ref{Def:HilbertBarron}, for every admissible extension
$F_e$, there holds
\begin{equation}
\label{eq:Fourier_moment_bound}
\|F\|_{\mathbb B_\tau([-1,1]^n;\mathbb X)}
\le M_0(F_e)+M_\tau(F_e).
\end{equation}
Here, both moments are evaluated on the same extension $F_e$.
The following lemma bounds the Barron norm of a product in terms
of these moments and will be used in
\eqref{eq:proof-product-moment-bound}.

\begin{lemma}\label{lem:compatible-product}
Let $\tau>0$, and let
$m,n\in\mathbb N_0$. Suppose that
$h:U_m=[-1,1]^m\to \mathbb X$ and $g:U_n=[-1,1]^n\to\R$ admit integrable
extensions $h_e:\R^m\to \mathbb X_{\C}$ and
$g_e:\R^n\to\C$ satisfying
$M_0(h_e)+M_\tau(h_e)<\infty$
and $M_0(g_e)+M_\tau(g_e)<\infty$.
Then
\begin{equation*}
\|h(\by_m)g(\by_n)\|_{\mathbb B_\tau(U_{m+n};\mathbb X)}
\le M_0(h_e)M_0(g_e)+ 2^{(\tau-1)_+} (M_\tau(h_e)M_0(g_e)
+  M_0(h_e)M_\tau(g_e)).
\end{equation*}
Here $(\tau-1)_+:=\max\{\tau-1,0\}$. If $m=0$ or $n=0$, the corresponding factor is constant. 
We interpret its zeroth moment as its absolute value or norm
and its moment of order $\tau$ as zero.
\end{lemma}
\begin{proof}
Set $H_e(\by_m,\by_n):=h_e(\by_m)g_e(\by_n)$. 
Since the variables are disjoint, there holds $H_e\in L^1(\R^{m+n};\mathbb X_{\C})$, 
and Fubini's theorem gives
$\widehat H_e(\xi,\eta)=\widehat h_e(\xi)\widehat g_e(\eta)$.
Since $g_e$ is scalar valued, we have
$\|\widehat H_e(\xi,\eta)\|_{\mathbb X_{\C}}
=\|\widehat h_e(\xi)\|_{\mathbb X_{\C}}
|\widehat g_e(\eta)|$.
Consequently, Tonelli's theorem gives
\[
\begin{aligned}
M_0(H_e)
&=\int_{\R^m}\int_{\R^n}
\|\widehat h_e(\xi)\|_{\mathbb X_{\C}}
|\widehat g_e(\eta)|\,\dd\eta\,\dd\xi\\
&=\left(\int_{\R^m}\|\widehat h_e(\xi)\|_{\mathbb X_{\C}}\,\dd\xi\right)
\left(\int_{\R^n}|\widehat g_e(\eta)|\,\dd\eta\right)
=M_0(h_e)M_0(g_e).
\end{aligned}
\]
For the moment of order $\tau$,  
using $\|(\xi,\eta)\|_1^\tau=(\|\xi\|_1+\|\eta\|_1)^\tau \leq  2^{(\tau-1)_+} (\|\xi\|_1^\tau+\|\eta\|_1^\tau)$, we obtain
\[
\begin{aligned}
M_\tau(H_e)
&\le 2^{(\tau-1)_+} \int_{\R^m}\int_{\R^n}
(\|\xi\|_1^\tau+\|\eta\|_1^\tau)
\|\widehat h_e(\xi)\|_{\mathbb X_{\C}}
|\widehat g_e(\eta)|\,\dd\eta\,\dd\xi\\
&=2^{(\tau-1)_+}
\left(\int_{\R^m}\|\xi\|_1^\tau
\|\widehat h_e(\xi)\|_{\mathbb X_{\C}}\,\dd\xi\right)
\left(\int_{\R^n}|\widehat g_e(\eta)|\,\dd\eta\right)\\
&\quad + 2^{(\tau-1)_+}
\left(\int_{\R^m}\|\widehat h_e(\xi)\|_{\mathbb X_{\C}}\,\dd\xi\right)
\left(\int_{\R^n}\|\eta\|_1^\tau
|\widehat g_e(\eta)|\,\dd\eta\right)\\
&= 2^{(\tau-1)_+} \bigl[M_\tau(h_e)M_0(g_e)+M_0(h_e)M_\tau(g_e)\bigr].
\end{aligned}
\]
These estimates show that $H_e$ is an admissible extension of
$h(\by_m)g(\by_n)$. Applying
\eqref{eq:Fourier_moment_bound} yields
the desired result.
\end{proof}

The following two lemmas provide dimension-dependent Fourier
moment bounds for extensions of Hilbert-valued holomorphic
functions
(Lemma \ref{lem:uniform-head-moments}) and dimension-independent
bounds for extensions of monomials
(Lemma \ref{lem:uniform-tail-moments}).

\begin{lemma}\label{lem:uniform-head-moments}
Let $n\in\mathbb N$, $r>0$, and $0<\delta\le\min\{1,r/4\}$ be fixed. Suppose that
$h:\bbU_r\to \mathbb X_{\C}$ is holomorphic  and satisfies $\| h \|_{L^{\infty}(\bbU_r;\mathbb X_{\C})}< \infty$, where
\[
\bbU_r:=\{\bz\in\C^n:
\operatorname{dist}(z_j,[-1,1])<r,\ j=0,\ldots,n-1\}.
\]
For every $\tau>0$, set $q:=\lceil\tau\rceil+n+1$. Then there exists
a compactly supported extension $h_e\in C^q(\R^n;\mathbb X_{\C})$
of $h|_{U_n}$ satisfying
\begin{equation}\label{eq:uniform-head-moments}
M_s(h_e)\le D_{1,\delta} (D_{2,\delta}q)^{s+n} \| h \|_{L^{\infty}(\bbU_r;\mathbb X_{\C})},
\end{equation}
simultaneously for all $0\le s\le\tau$.
Here $D_{1,\delta}:=2[4(1+\delta)]^n$ and $D_{2,\delta}:=\frac{n}{\delta}$ depend only on $n$ and $\delta$. 
\end{lemma}
\begin{proof}
Let $\chi_{q,\delta}\in C_c^q(\R)$ be the cutoff
function from Lemma~\ref{lem_cutoff}, which equals one on
$[-1,1]$ and is supported in $[-1-\delta,1+\delta]$.
Define
\[
h_e(\by_n):=h(\by_n)\prod_{j=0}^{n-1}\chi_{q,\delta}(y_j)
\quad\text{on }[-1-\delta,1+\delta]^n,
\]
and extend it by zero outside this cube.
Then $h_e\in C_c^q(\R^n;\mathbb X_{\C})$
and $h_e=h$ on $U_n$.

For every $\by_n$ in the support, the closed polydisc
$\{\bz:|z_j-y_j|\le\delta,\ j=0,\ldots,n-1\}$ lies in $\bbU_r$,
since $\operatorname{dist}(z_j,[-1,1])\le2\delta<r$.
Cauchy's integral formula in the $j$th variable therefore gives
\begin{equation}
\begin{split}
\label{eq:y_L_inf}
\|\partial_j^k h(\by_n)\|_{\mathbb X_{\C}}
&=
\Bigl\|\frac{k!}{2\pi i} \int_{|z_j - y_j| =\delta} \frac{h(y_0,\cdots,z_j, \cdots,y_{n-1})}{ (z_j - y_j)^{k+1}} dz_j \Bigr\|_{\mathbb X_{\C}} \\
&\le k!\delta^{-k} \| h \|_{L^{\infty}(\bbU_r;\mathbb X_{\C})}
\le \Bigl( \frac{q}{\delta}\Bigr)^k \| h \|_{L^{\infty}(\bbU_r;\mathbb X_{\C})},
\qquad 0\le k\le q.
\end{split}
\end{equation}
By \eqref{eq_chi_R} in Lemma~\ref{lem_cutoff},  
$\|\chi_{q,\delta}\|_{L^1}\le 2(1+\delta)$. Moreover,
\eqref{eq:cutoff-delta-L1}, $q+1\le2q$, and $\delta\le1$ imply
\begin{equation}
\label{eq:X_L1}
\|\chi_{q,\delta}^{(k)}\|_{L^1}
\le2\left[\frac{2(q+1)}{\delta}\right]^{k-1}
\le 2(1+\delta)  \Bigl( \frac{4q}{\delta}\Bigr)^{k},
\qquad 1\le k\le q.
\end{equation}
The last bound also holds for $k=0$ by the preceding
$L^1$ estimate.
Set $C_\delta:=[2(1+\delta)]^n\|h\|_{L^\infty(\bbU_r;\mathbb X_{\C})}$.
Leibniz's formula, \eqref{eq:y_L_inf}, \eqref{eq:X_L1}, and the binomial theorem yield
\begin{subequations}
\begin{align}
\|h_e\|_{L^1(\R^n;\mathbb X_{\C})}
&\le C_\delta, \label{eq:he_L1}\\
\|\partial_j^q h_e\|_{L^1(\R^n;\mathbb X_{\C})}
&\le C_\delta\sum_{k=0}^q\binom{q}{k}
 \Bigl( \frac{q}{\delta}\Bigr)^k \Bigl( \frac{4q}{\delta}\Bigr)^{q-k}
=C_\delta \Bigl( \frac{5q}{\delta}\Bigr)^q. \label{eq:partial_he_L1}
\end{align}
\end{subequations}
For $\xi\ne0$, choose a coordinate with $|\xi_j|\ge\|\xi\|_1/n$.
Since $h_e\in C_c^q(\R^n;\mathbb X_{\C})$, integration by
parts $q$ times in $y_j$ gives
\[
\widehat h_e(\xi)
=\frac{1}{(2\pi i\xi_j)^q}
\int_{\R^n}\partial_j^q h_e(\by_n)
e^{-2\pi i\langle\by_n,\xi\rangle}\,\dd\by_n,
\]
where all boundary terms vanish by compact support. Using \eqref{eq:partial_he_L1}, we obtain
\[
\|\widehat h_e(\xi)\|_{\mathbb X_{\C}}
\le\frac{\|\partial_j^q h_e\|_{L^1(\R^n;\mathbb X_{\C})}}
{(2\pi|\xi_j|)^q}
\le C_\delta\left(\frac{5nq}{2\pi\delta\|\xi\|_1}\right)^q.
\]
On the other hand, the defining Fourier integral gives
$\|\widehat h_e(\xi)\|_{\mathbb X_{\C}}
\le\|h_e\|_{L^1(\R^n;\mathbb X_{\C})}\le C_\delta$.
Combining these two bounds and using $5/(2\pi)<1$ and
$D_{2,\delta}=n/\delta$, we obtain
\[
\|\widehat h_e(\xi)\|_{\mathbb X_{\C}}
\le C_\delta\min\left\{1,
\left(\frac{5nq}{2\pi\delta\|\xi\|_1}\right)^q\right\}
\le C_\delta\min\left\{1,
\left(\frac{D_{2,\delta}q}{\|\xi\|_1}\right)^q\right\}.
\]
The $\ell^1$ ball of radius $t$ in $\R^n$ has volume
$2^n t^n/n!$. Integrating the preceding bound and splitting the
radial integral at $t=D_{2,\delta}q$ gives, for $0\le s\le\tau$,

\[
\begin{aligned}
M_s(h_e)
&\le C_\delta\int_{\R^n}\|\xi\|_1^s
\min\left\{1,\left(\frac{D_{2,\delta}q}{\|\xi\|_1}\right)^q\right\}
\,\dd\xi\\
&=\frac{2^nC_\delta}{(n-1)!}
\int_0^\infty t^{s+n-1}
\min\left\{1,\left(\frac{D_{2,\delta}q}{t}\right)^q\right\}
\,\dd t\\
&=\frac{2^nC_\delta}{(n-1)!}
(D_{2,\delta}q)^{s+n}
\left(\frac1{s+n}+\frac1{q-s-n}\right).
\end{aligned}
\]
Since $s+n\ge1$ and $q-s-n\ge q-\tau-n\ge1$, the sum in parentheses is at
most two. Using $(n-1)!\ge1$ and
$2^{n+1}C_\delta=D_{1,\delta}\|h\|_{L^\infty(\bbU_r;\mathbb X_{\C})}$ proves
\eqref{eq:uniform-head-moments}.
\end{proof}

Next, we establish Fourier moment bounds for extensions of
the monomial $\by_n^{\bm\nu}$, where
$\by_n=(y_0,\ldots,y_{n-1})$ and
$\bm\nu\in\mathbb N_0^n$ is a multi-index.
Recall that
$\|\bm\nu\|_0:=|\operatorname{supp}(\bm\nu)|$,
$\|\bm\nu\|_1:=\sum_{j=0}^{n-1}\nu_j$,
$\bm\nu!=\prod_{j}\nu_j!$,
and
$\by^{\bm\nu}
=\prod_{j}y_j^{\nu_j}$.

\begin{lemma} 
\label{lem:uniform-tail-moments}
For every $n\ge1$ and $\tau>0$, there exists a family of
real-valued integrable extensions
$\{g_e^{\bnu}\}_{\bnu\in\mathbb N_0^n}$ of the monomials
$\by_n^{\bnu}$ on $U_n=[-1,1]^n$ such that
\begin{align}
&\|g_e^{\bnu}\|_{L^1(\R^n)}
\le\|g_e^{\bm0}\|_{L^1(\R^n)}B_1^{\|\bnu\|_1} < \infty,
\label{eq:polarization-monomial-L1}\\
&M_s(g_e^{\bnu})
\le B_2\,B_3^s(\tau+4)^{s+1}B_1^{\|\bnu\|_1}
\qquad \textrm{simultaneously~for~all~} 0\le s\le\tau.
\label{eq:polarization-monomial-moments}\end{align}
Here $B_1=2e$, $B_2=160$, and $B_3=8$ are universal
constants. The quantity $\|g_e^{\bm0}\|_{L^1(\R^n)}$
may depend on $n$ and $\tau$, whereas the Fourier
moment bound is independent of $n$.
\end{lemma}

\begin{proof}
We construct the desired extension using one-dimensional
monomial extensions and an integrable extension of the
constant function $1$ on $U_n$.
Set $m:=\|\bm\nu\|_1$ and $q:=\lceil\tau\rceil+2$.

\noindent\textbf{Step 1. Extensions of univariate monomials and constants.}
For $m\in\mathbb N_0$, the polynomial $z^m$ is holomorphic on
$\bbU_1
:=\{z\in\C:\operatorname{dist}(z,[-1,1])<1\}$
and satisfies $|z^m|\le2^m$ there.
Applying the scalar case of
Lemma~\ref{lem:uniform-head-moments} with
$n=1$, $r=1$, and $\delta=1/4$ yields the extension
$p_e^{m,q}(t):=t^m\chi_{q,1/4}(t) \in C_c^q(\R;\R)$,
which satisfies $\|p_e^{m,q}\|_{L^\infty(\R)} \leq 2^m$ and, simultaneously for all $0\le s\le\tau$,
\begin{equation}\label{eq:tail-one-dimensional-moments}
M_s(p_e^{m,q})\le10(4q)^{s+1}2^m.
\end{equation}

For the constant function $1$ on $U_n$ ($\bm\nu = \bm0$), the construction in \cite[Lemma 4]{2023SiegelXu}, adapted
to the weight $1+\|\xi\|_1^\tau$ yields
a real-valued integrable extension $g_e^{\bm 0}$ of $1$ satisfying
$M_0(g_e^{\bm 0})  + M_\tau(g_e^{\bm 0}) \le2$. We fix this extension for all $\bnu$.
Using $\|\xi\|_1^s\le1+\|\xi\|_1^\tau$,
we obtain
\begin{equation}\label{lem:uniform-tail-moments_eq1}
M_s(g_e^{\bm0})\le2,
\qquad 0\le s\le\tau.
\end{equation}
\noindent\textbf{Step 2. Polarization and construction of the extension.}
The case $\bnu=\bm0$ follows from Step~1.
Henceforth, assume $m=\|\bnu\|_1\ge1$.
Choose $j_1,\ldots,j_m\in\{0,\ldots,n-1\}$ so that
each index $j$ occurs $\nu_j$ times. 
For $\boldsymbol\varepsilon\in\{-1,1\}^m$, 
the real polarization identity~\cite[Appendix, (A.4)]{thomas2014polarization}
gives
\begin{equation}\label{eq:discrete-monomial-polarization}
\by_n^{\bnu}
=\frac{m^m}{2^m m!}\sum_{\boldsymbol\varepsilon\in\{-1,1\}^m}
\omega_{\boldsymbol\varepsilon}
\langle\mathbf a_{\boldsymbol\varepsilon},\by_n\rangle^m \quad \textrm{with} \quad \mathbf a_{\boldsymbol\varepsilon}
:=\frac1m\sum_{\ell=1}^m\varepsilon_\ell\mathbf e_{j_\ell}, \quad \omega_{\boldsymbol\varepsilon}:=\prod_{\ell=1}^m\varepsilon_\ell.
\end{equation}
Here $\mathbf e_j$ is the $j$th coordinate vector in $\R^n$ and $\|\mathbf a_{\boldsymbol\varepsilon}\|_1\le1$.
Define the explicit extension
\begin{equation}\label{eq:discrete-monomial-extension}
g_e^{\bnu}(\by_n)
:=\frac{m^m}{2^m m!} g_e^{\bm0}(\by_n)
\sum_{\boldsymbol\varepsilon\in\{-1,1\}^m}
\omega_{\boldsymbol\varepsilon}
p_e^{m,q}(\langle\mathbf a_{\boldsymbol\varepsilon},\by_n\rangle).
\end{equation}
On $U_n$, the arguments of $p_e^{m,q}$ belong to $[-1,1]$
and $g_e^{\bm0}=1$, so
\eqref{eq:discrete-monomial-polarization} gives
$g_e^{\bnu}=\by_n^{\bnu}$. Since
$m^m/m!\le e^m$, the finite sum also gives
\[
\|g_e^{\bnu}\|_{L^1(\R^n)}
\le2^m  \frac{m^m}{2^m m!}\|g_e^{\bm0}\|_{L^1(\R^n)}
\|p_e^{m,q}\|_{L^\infty(\R)}
\le(2e)^m\|g_e^{\bm0}\|_{L^1(\R^n)}.
\]

\noindent\textbf{Step 3. Fourier moments.}
For $\|\mathbf a\|_1\le1$, write
$g_{\mathbf a}(\by_n):=g_e^{\bm0}(\by_n)
p_e^{m,q}(\langle\mathbf a,\by_n\rangle)$.
Fourier inversion and Fubini's theorem, justified by
$\|g_e^{\bm0}\|_{L^1(\R^n)}M_0(p_e^{m,q})<\infty$, give
\[
\widehat g_{\mathbf a}(\xi)
=\int_{\R^n}g_e^{\bm0}(\by_n)
\Bigl(\int_{\R}\widehat{p_e^{m,q}}(\zeta)
e^{2\pi i\zeta\langle\mathbf a,\by_n\rangle}\dd\zeta\Bigr) e^{-2\pi i \xi \cdot \by_n} \dd\by_n
=\int_{\R}\widehat{p_e^{m,q}}(\zeta)
\widehat{g_e^{\bm0}}(\xi-\zeta\mathbf a)\,\dd\zeta.
\]
Combining \eqref{eq:tail-one-dimensional-moments}
and \eqref{lem:uniform-tail-moments_eq1}, we obtain $M_0(g_{\mathbf a})\le M_0(g_{e}^{\bm0})M_0(p_e^{m,q})\le80q\,2^m$. For $0<s\le\tau$,
Tonelli's theorem and
$\|\eta+\zeta\mathbf a\|_1^s
\le2^{(s-1)_+}(\|\eta\|_1^s+|\zeta|^s)$ yield
\[
\begin{aligned}
M_s(g_{\mathbf a})
&\le2^{(s-1)_+}\bigl[
M_s(g_e^{\bm0})M_0(p_e^{m,q})
+M_0(g_e^{\bm0})M_s(p_e^{m,q})\bigr]\\
&\le80q\,2^m\,2^{(s-1)_+}\bigl[1+(4q)^s\bigr]
\le160q(8q)^s2^m.
\end{aligned}
\]
The last bound also holds for $s=0$. Applying it to the
finite sum in \eqref{eq:discrete-monomial-extension} gives
\[
M_s(g_e^{\bnu})
\le2^m \frac{m^m}{2^m m!}\,160q(8q)^s2^m
\le160q(8q)^s(2e)^m,
\qquad 0\le s\le\tau.
\]
Since $q\le\tau+4$, this proves
\eqref{eq:polarization-monomial-moments}.
\end{proof}

%% file: Chapters/3_Theory_coercive.tex

\section{Barron regularity under uniform coercivity}
\label{sec:BarronNormResults}

In this section, we establish spectral Barron regularity for
the parameter-to-solution map $\by_N\mapsto u_N(\by_N)$ defined
by \eqref{eq:elliptic-pde_trunc}, with norm bounds uniform in $N$.
Our argument combines the parametric holomorphy techniques
in~\cite{CohenDeVoreSchwab2011} with the Fourier moment estimates
from Section~\ref{sec:harmonic}.
We first establish a holomorphic extension in
Subsection~\ref{sec:poly}, then derive a partial Taylor expansion
and coefficient estimates in Subsection~\ref{ssec:Taylor}, and finally prove
Theorem~\ref{the:ridge-map-barron} in
Subsection~\ref{ssec:proof-barron}.

\subsection{Holomorphic extension near the real cube}
\label{sec:poly}

Fix $N\ge1$ and define the following complex neighborhood of
$U_N=[-1,1]^N$:
\begin{equation}
\label{eq:Ucal}
\bbU:=\left\{\bz_N\in\C^N:
\inf_{\by_N\in U_N}C_{\mathfrak b}
\sum_{j<N}|b_j|\,|z_j-y_j|<\frac{C_{\star}}{2}\right\}.
\end{equation}
The distance to the real cube is weighted by the coefficient
amplitudes $|b_j|$. 

Extend the affine coefficient parameterization to complex
parameters. For $\bz_N\in\C^N$, let
$\mathcal A(\bz_N)\in L(\mathbb V_{\C},\mathbb V_{\C}')$
be the bounded linear operator associated with the complexified variational form:
\begin{equation}\label{eq:parameter-operator}
\langle \mathcal A(\bz_N)v,w\rangle_{\mathbb V_{\C}',\mathbb V_{\C}}
:=\mathfrak b(a_N(\cdot;\bz_N);v,w),
\qquad v,w\in \mathbb V_{\C}.
\end{equation}
Here $\mathbb V_{\C}'$ denotes the space of continuous
conjugate-linear functionals on $\mathbb V_{\C}$, and the
complexified form is linear in $v$ and conjugate-linear in $w$.
The following proposition establishes the holomorphic
extension of $u_N$ and uniform bounds for $\mathcal A$ and its inverse
on $\bbU$.

\begin{proposition}\label{prop:holomorphic-extension}
Under Assumption~\ref{ass:ellipticity}, the map $\mathcal A$ is complex-affine and
hence holomorphic in the operator norm on $\C^N$, with
\begin{equation}\label{eq:parameter-operator-affine}
\begin{gathered}
\mathcal A(\bz_N)=\mathcal A(\bm0)+\sum_{j=0}^{N-1}z_j\mathcal A_j',
\quad \mathcal A_j':=\partial_{z_j}\mathcal A,\quad \textrm{and}\quad
\|\mathcal A_j'\|_{L(\mathbb V_{\C},\mathbb V_{\C}')}
\le C_{\mathfrak b}|b_j|,
\end{gathered}
\end{equation}
where the operators $\mathcal A_j'$ are independent of $\bz_N$, $j=0,\ldots,N-1$.

For every $\bz_N\in\bbU$, the operator $\mathcal A(\bz_N)$
is invertible,  and the map
$\bz_N\mapsto \mathcal A(\bz_N)^{-1}$ is holomorphic on $\bbU$ in the operator norm. Moreover,
\begin{equation}\label{eq:parameter-operator-bounds}
\sup_{\bz_N\in\bbU}
\|\mathcal A(\bz_N)\|_{L(\mathbb V_{\C},\mathbb V_{\C}')}
\le C^{\star}+\frac{C_{\star}}{2},
\qquad
\sup_{\bz_N\in\bbU}
\|\mathcal A(\bz_N)^{-1}\|_{L(\mathbb V_{\C}',\mathbb V_{\C})}
\le\frac{2}{C_{\star}}.
\end{equation}
Consequently, the parameter-to-solution map $u_N:U_N\to \mathbb V$
defined by \eqref{eq:elliptic-pde_trunc} admits a holomorphic
extension $u_N: \bbU\to \mathbb V_{\C}$ given by $u_N(\bz_N)=\mathcal A(\bz_N)^{-1}f$,
with
\begin{equation}\label{eq:energyEstimate}
\sup_{\bz_N\in\bbU}\|u_N(\bz_N)\|_{\mathbb V_{\C}}\le2C^f,
\qquad C^f:=(C_{\star})^{-1}\|f\|_{\mathbb V'}.
\end{equation}
\end{proposition}

\begin{proof}
The affine property in
Assumption~\ref{ass:ellipticity-B2} gives the decomposition in
\eqref{eq:parameter-operator-affine}. 
The coefficient-difference estimate in Assumption~\ref{ass:ellipticity-A1}, 
extended to complex coefficients and arguments by
Assumption~\ref{ass:ellipticity-B2}, together with
\eqref{eq:target-coefficient-expansion} and
$\|\psi_j\|_{L^\infty(\Omega)}\le1$, implies for all $\bz_N,\widetilde\bz_N\in\C^N$
\begin{equation}
\label{eq:parameter-operator-difference}
\|\mathcal A(\bz_N)-\mathcal A(\widetilde\bz_N)\|_{
L(\mathbb V_{\C},\mathbb V_{\C}')}
\le C_{\mathfrak b}\|a_N(\cdot; \bz_N) - a_N(\cdot; \widetilde\bz_N)\|_{L^{\infty}(\Omega)}
\le C_{\mathfrak b}\sum_{j<N}|b_j|\,|z_j-\widetilde z_j|.
\end{equation}
Varying only the $j$th coordinate yields the bound \eqref{eq:parameter-operator-affine} on $\mathcal A_j'$.
Since $\mathcal A$ is complex-affine, it is holomorphic in the operator norm.

The uniform real coercivity in Assumption~\ref{ass:ellipticity-A2} is preserved under complexification of the real symmetric form. Indeed, if $v=v_1+i v_2\in \mathbb V_{\C}$
and $\by_N\in U_N$, then
\[
\langle \mathcal A(\by_N)v,v\rangle
=\mathfrak b(a_N(\cdot;\by_N);v_1,v_1)
+\mathfrak b(a_N(\cdot;\by_N);v_2,v_2)
\ge C_{\star}\|v\|_{\mathbb V_{\C}}^2.
\]
The imaginary cross terms cancel by symmetry. 
Given $\bz_N\in\bbU$, choose $\by_N\in U_N$ such that
$C_{\mathfrak b}\sum_{j<N}|b_j|\,|z_j-y_j|<C_{\star}/2$.
Using \eqref{eq:parameter-operator-difference}, we obtain $\|\mathcal A(\bz_N)-\mathcal A(\by_N)\|_{L(\mathbb V_{\C},\mathbb V_{\C}')}\leq \frac{C_{\star}}{2}$.
Since complexification preserves operator norms between
real Hilbert spaces, Assumption~\ref{ass:ellipticity-A2} gives
$\|\mathcal A(\by_N)\|_{
L(\mathbb V_{\C},\mathbb V_{\C}')}\le C^{\star}$.
The triangle inequality therefore gives the first bound
in \eqref{eq:parameter-operator-bounds}.
Moreover, the perturbation estimate \eqref{eq:parameter-operator-difference} gives
\[
\begin{aligned}
\operatorname{Re}\langle \mathcal A(\bz_N)v,v\rangle
\ge\langle \mathcal A(\by_N)v,v\rangle
-\|\mathcal A(\bz_N)-\mathcal A(\by_N)\|_{L(\mathbb V_{\C},\mathbb V_{\C}')}\,\|v\|_{\mathbb V_{\C}}^2
\ge\frac{C_{\star}}{2}\|v\|_{\mathbb V_{\C}}^2,
\qquad v\in \mathbb V_{\C}.
\end{aligned}
\]
The Lax-Milgram theorem 
implies that
$\mathcal A(\bz_N)$ is invertible and yields the second bound in \eqref{eq:parameter-operator-bounds}.

To prove holomorphy of the parameter-to-solution map, we first establish holomorphy of the inverse map $\bz_N\mapsto \mathcal A(\bz_N)^{-1}$. Fix $\bz_N^0\in\bbU$
and set
$T(\bz_N):=\mathcal A(\bz_N^0)^{-1}
\bigl(\mathcal A(\bz_N)-\mathcal A(\bz_N^0)\bigr)$.
Since $T$ is continuous and $T(\bz_N^0)=0$, there is a
neighborhood of $\bz_N^0$ contained in $\bbU$ on which
$\|T(\bz_N)\|_{L(\mathbb V_{\C},\mathbb V_{\C})}\le1/2$.
On this neighborhood, the Neumann series gives
\begin{equation}
\mathcal A(\bz_N)^{-1}
= \Bigl(\mathcal A(\bz_N^0)\bigl(I + T(\bz_N)\bigr)\Bigr)^{-1} = \sum_{j=0}^{\infty}(-T(\bz_N))^j \mathcal A(\bz_N^0)^{-1}.
\end{equation}
Each term is holomorphic, and the series converges uniformly
in the operator norm on this neighborhood.
As $\bz_N^0$ is arbitrary, $\mathcal A^{-1}$ is holomorphic on $\bbU$.
Thus, $u_N(\bz_N):=\mathcal A(\bz_N)^{-1}f$ is holomorphic
and agrees with the real solution on $U_N$ by uniqueness,
which yields \eqref{eq:energyEstimate}.
\end{proof}

\begin{lemma}
   \label{lem:G_continuity}
Under Assumption~\ref{ass:ellipticity}, for $g,\widetilde g\in\mathbb B^{\bb}(\Omega)$, there holds
\begin{equation}\label{eq:solution-lipschitz}
\begin{aligned}
\|\mathcal G^\dagger(g)-\mathcal G^\dagger(\widetilde g)\|_{\mathbb H}
&\le C_S
\|\mathcal G^\dagger(g)-\mathcal G^\dagger(\widetilde g)\|_{\mathbb V}  \le\frac{C_SC_{\mathfrak b}C^f}{C_{\star}}
\|g-\widetilde g\|_{L^\infty(\Omega)}.
\end{aligned}
\end{equation}
\end{lemma}
\begin{proof}
Subtracting the corresponding variational equations and using
the coefficient-difference estimate in
Assumption~\ref{ass:ellipticity-A1},
the uniform coercivity in
Assumption~\ref{ass:ellipticity-A2},
the energy bound $\|\mathcal G^\dagger(\widetilde g)\|_{\mathbb V}\le C^f$,
and the continuous embedding $\mathbb V\hookrightarrow \mathbb H$, we obtain \eqref{eq:solution-lipschitz}.
\end{proof}


\subsection{Partial Taylor expansion
and coefficient estimates}
\label{ssec:Taylor}
In this subsection, we derive a partial Taylor expansion
of the parameter-to-solution map in the tail variables
and establish coefficient estimates for the
proof of Barron regularity.
Following the framework in
\cite[Section~3]{CohenDeVoreSchwab2011}, we split the
parameter indices into a finite set $\mathfrak{J}$ and a tail $\mathfrak{J}^c$,
with the split chosen using the summability of $\bb$.
One key difference is that we only expand the solution with respect to the tail index set $\mathfrak{J}^c$.
Let $B_1,B_2,B_3$ be the universal constants in
Lemma~\ref{lem:uniform-tail-moments}. Since $\bb\in\ell^1$, choose
$N_b\in\N$, independently of $N$ and $\tau$, such that
\begin{equation}\label{eq:ridge-tail-smallness}
B_{\bm\kappa}:=B_1\sum_{j\ge N_b} \kappa_j \le\frac12,
\quad \textrm{with}\quad \kappa_j:=\frac{2C_{\mathfrak b}|b_j|}{C_{\star}},
\qquad \bm\kappa :=(\kappa_0, \kappa_1, \cdots ).
\end{equation}
For each $N\ge1$, write
\[
N_0:=\min\{N,N_b\},\qquad
\mathfrak{J}:=\{0,\ldots,N_0-1\},\qquad \mathfrak{J}^c:=\{N_0,\ldots,N-1\}.
\]
Then $\mathfrak{J}^c=\varnothing$ if $N\le N_b$; otherwise,
$\mathfrak{J}^c\subseteq\{j\ge N_b\}$.

To apply Lemma~\ref{lem:uniform-head-moments} to the
variables indexed by $\mathfrak{J}$, define the complex neighborhood
\begin{equation}
\label{U_r_E}
\bbU^{N_0}_r:=\bigl\{\bz\in\C^{N_0}:
\operatorname{dist}(z_j,[-1,1])<r,\ j\in \mathfrak{J}\bigr\},\quad r:=\min\left\{1,\frac{C_{\star}}{4C_{\mathfrak b}\sum_{j<N_b}|b_j|}\right\}
\end{equation}
We set $r=1$ if the denominator vanishes.
We have the following partial Taylor expansion:

\begin{lemma}\label{lem:tail-taylor-coeff}
Under Assumption~\ref{ass:ellipticity}, and with the
notation above, for every $N\ge1$, there exist holomorphic functions
$t_{\bnu}:\bbU^{N_0}_r\to \mathbb V_{\C}$ such that
\begin{equation}\label{eq:partial-taylor-expansion}
u_N(\bz_{\mathfrak{J}},\by_{\mathfrak{J}^c})
=\sum_{\bnu\in\N_0^{\mathfrak{J}^c}}t_{\bnu}(\bz_{\mathfrak{J}})\by_{\mathfrak{J}^c}^{\bnu}.
\end{equation}
The series converges absolutely and uniformly in the
$\mathbb V_{\C}$-norm on $\bbU^{N_0}_r\times[-1,1]^{\mathfrak{J}^c}$. Moreover,
\begin{equation}\label{eq:tail-taylor-coeff-bound}
\|t_{\bnu}\|_{L^\infty(\bbU^{N_0}_r;\mathbb V_{\C})}
\le S_{\bnu}:=2C^f\frac{\|\bnu\|_1!}{\bnu!}\bm\kappa^{\bnu},
\qquad
\sum_{\bnu\in\N_0^{\mathfrak{J}^c}}S_{\bnu}B_1^{\|\bnu\|_1}
\le S_*:=\frac{2C^f}{1-B_{\bm\kappa}}.
\end{equation}
Here $B_{\bm\kappa}$ and $\bm\kappa$ are defined in
\eqref{eq:ridge-tail-smallness}. If $\mathfrak{J}^c$ is empty, the sum contains only $\bnu=0$ and
$\by_{\mathfrak{J}^c}^{\bnu}=1$.
\end{lemma}

\begin{proof}
We apply Proposition~\ref{prop:holomorphic-extension}
on a suitable product domain and then estimate the
Taylor coefficients.

\textbf{Step 1. A product domain of holomorphy.}
We show that the product domain $\bbU^{N_0}_r\times\{\bz_{\mathfrak{J}^c}:\|\bz_{\mathfrak{J}^c}\|_\infty < 2\}$ is contained in $\bbU$.
For $\bz_{\mathfrak{J}}\in\bbU^{N_0}_r$ and $\|\bz_{\mathfrak{J}^c}\|_\infty < 2$, choose
$\by_{\mathfrak{J}}\in[-1,1]^{\mathfrak{J}}$ with $|z_j-y_j|<r$ for $j\in \mathfrak{J}$,
and set $\by_{\mathfrak{J}^c}=\bm0$. Using \eqref{U_r_E}, \eqref{eq:ridge-tail-smallness},
and $B_1=2e>2$, we obtain
\[
\begin{aligned}
C_{\mathfrak b}\sum_{j<N}|b_j|\,|z_j-y_j|
\le rC_{\mathfrak b}\sum_{j\in \mathfrak{J}}|b_j|
   +2C_{\mathfrak b}\sum_{j\in \mathfrak{J}^c}|b_j| \le\frac{C_{\star}}{4}+C_{\star}\sum_{j\in \mathfrak{J}^c}\kappa_j
\le\frac{C_{\star}}{4}+\frac{C_{\star}}{2B_1}<\frac{C_{\star}}{2}.
\end{aligned}
\]
Thus
$\bbU^{N_0}_r\times\{\bz_{\mathfrak{J}^c}:\|\bz_{\mathfrak{J}^c}\|_\infty < 2\}\subset\bbU$.
Proposition~\ref{prop:holomorphic-extension} supplies the
holomorphic solution on $\bbU$, with
$\|u_N\|_{\mathbb V_{\C}}\le2C^f$ on this product set.
Define $\mathcal A_{\mathfrak{J}}(\bz_{\mathfrak{J}}):=\mathcal A(\bz_{\mathfrak{J}},\bm0)$.
By the same proposition, $\mathcal A_{\mathfrak{J}}^{-1}$ is holomorphic on
$\bbU^{N_0}_r$ and $\|\mathcal A_{\mathfrak{J}}^{-1}\|_{
L(\mathbb V_{\C}',\mathbb V_{\C})}\le2/C_{\star}$ there.

\noindent\textbf{Step 2. Partial Taylor expansion.}
For $\bz_{\mathfrak{J}}\in\bbU^{N_0}_r$,  define
\[
t_{\bnu}(\bz_{\mathfrak{J}})
:=\frac1{\bnu!}
\partial_{\bz_{\mathfrak{J}^c}}^{\bnu}u_N(\bz_{\mathfrak{J}},\bm0),
\qquad \bnu\in\N_0^{\mathfrak{J}^c}.
\]
Since $u_N$ is jointly holomorphic in $(\bz_{\mathfrak{J}},\bz_{\mathfrak{J}^c})$,
its partial derivatives in $\bz_{\mathfrak{J}^c}$ remain holomorphic.
Evaluating these derivatives at $\bz_{\mathfrak{J}^c}=\bm0$ shows that
each $t_{\bnu}$ is holomorphic on $\bbU^{N_0}_r$.
By Step~1, the Taylor expansion in the tail variables
about $\bz_{\mathfrak{J}^c}=\bm0$ represents $u_N$ for
$\|\bz_{\mathfrak{J}^c}\|_\infty<2$, i.e., 
\[
u_N(\bz_{\mathfrak{J}},\by_{\mathfrak{J}^c})
=\sum_{\bnu\in\N_0^{\mathfrak{J}^c}}
t_{\bnu}(\bz_{\mathfrak{J}})\by_{\mathfrak{J}^c}^{\bnu},
\qquad
\bz_{\mathfrak{J}}\in\bbU^{N_0}_r,\quad \by_{\mathfrak{J}^c}\in[-1,1]^{\mathfrak{J}^c}.
\]

\noindent\textbf{Step 3. Coefficient estimates.}
By \eqref{eq:parameter-operator-affine}, the solution satisfies
\[
\left(\mathcal A_{\mathfrak{J}}(\bz_{\mathfrak{J}})+\sum_{j\in \mathfrak{J}^c}z_j\mathcal A_j'\right)
u_N(\bz_{\mathfrak{J}},\bz_{\mathfrak{J}^c})=f,
\]
where, for $j\in \mathfrak{J}^c$,
$\mathcal A_j'=\left.\partial_{z_j}\mathcal A(\bz_{\mathfrak{J}},\bz_{\mathfrak{J}^c})\right|_{\bz_{\mathfrak{J}^c}=\bm0}$
is the parameter-independent operator in \eqref{eq:parameter-operator-affine}.
Applying $\partial_{\bz_{\mathfrak{J}^c}}^{\bnu}$ to the operator equation,
evaluating at $\bz_{\mathfrak{J}^c}=\bm0$, and dividing by $\bnu!$ gives
\[
t_{\bm0}=\mathcal A_{\mathfrak{J}}^{-1}f,
\qquad
t_{\bnu}
=-\mathcal A_{\mathfrak{J}}^{-1}\sum_{j\in \mathfrak{J}^c:\nu_j>0}\mathcal A_j't_{\bnu-\mathbf e_j},
\qquad \bnu\ne\bm0,
\]
where $\mathbf e_j$ is the unit multi-index corresponding to $j$.
The bounds in Proposition~\ref{prop:holomorphic-extension} imply
$\sup_{\bz_{\mathfrak{J}}\in\bbU^{N_0}_r}\|\mathcal A_{\mathfrak{J}}(\bz_{\mathfrak{J}})^{-1}\mathcal A_j'\|_{
L(\mathbb V_{\C},\mathbb V_{\C})}\le \kappa_j$ and
$\|t_{\bm0}\|_{L^\infty(\bbU^{N_0}_r;\mathbb V_{\C})}\le2C^f$.
Induction on $m=\|\bnu\|_1\ge1$ therefore gives
\[
\begin{aligned}
\|t_{\bnu}\|_{L^\infty(\bbU^{N_0}_r;\mathbb V_{\C})}
&\le\sum_{j\in \mathfrak{J}^c:\nu_j>0}\kappa_j
\|t_{\bnu-\mathbf e_j}\|_{L^\infty(\bbU^{N_0}_r;\mathbb V_{\C})}\\
&\le2C^f\sum_{j\in \mathfrak{J}^c:\nu_j>0}\kappa_j
\frac{(m-1)!}{(\bnu-\mathbf e_j)!}
\bm\kappa^{\bnu-\mathbf e_j}
=2C^f\frac{m!}{\bnu!}\bm\kappa^{\bnu}
=S_{\bnu},
\end{aligned}
\]
where the equality uses $\sum_{j\in \mathfrak{J}^c}\nu_j=m$.
For $\by_{\mathfrak{J}}\in[-1,1]^{\mathfrak{J}}$,  the recurrence also shows that
$t_{\bnu}(\by_{\mathfrak{J}})\in \mathbb V$.
Finally, the multinomial formula yields
\begin{equation}
\label{eq:summability}
\sum_{\bnu\in\N_0^{\mathfrak{J}^c}}S_{\bnu}B_1^{\|\bnu\|_1}
=2C^f\sum_{m=0}^{\infty}
\left(B_1\sum_{j\in \mathfrak{J}^c}\kappa_j\right)^m  \le\frac{2C^f}{1-B_{\bm\kappa}}=S_*.
\end{equation}
This proves \eqref{eq:tail-taylor-coeff-bound} with a bound
independent of $N$.
Since $B_1\ge1$ and $|\by_{\mathfrak{J}^c}^{\bnu}|\le1$ on
$[-1,1]^{\mathfrak{J}^c}$, the weighted summability estimate~\eqref{eq:summability} also proves
absolute uniform convergence of
\eqref{eq:partial-taylor-expansion} on
$\bbU^{N_0}_r\times[-1,1]^{\mathfrak{J}^c}$.

\end{proof}

\subsection{Proof of Theorem~\ref{the:ridge-map-barron}}
\label{ssec:proof-barron}

We combine the partial Taylor expansion~\eqref{eq:partial-taylor-expansion} in
Lemma~\ref{lem:tail-taylor-coeff} with the Fourier moment estimates
from Section~\ref{sec:harmonic}. The bound $N_0\le N_b$ and the
weighted coefficient estimate allow us to sum the extended Taylor
terms and obtain a Barron norm bound uniform in $N$. 

\begin{proof}
{Continuity of $\Gcal^{\dagger}$ 
follows from Lemma~\ref{lem:G_continuity}.}  
Fix $N\ge1$ and $\tau>0$, and retain the notation of
Subsection~\ref{ssec:Taylor}.

\noindent\textbf{Step 1. Extensions.}
Recall the partial Taylor expansion \eqref{eq:partial-taylor-expansion}.
\begin{equation}
u_N(\bz_{\mathfrak{J}},\by_{\mathfrak{J}^c})
=\sum_{\bnu\in\N_0^{\mathfrak{J}^c}}t_{\bnu}(\bz_{\mathfrak{J}})\by_{\mathfrak{J}^c}^{\bnu}.
\end{equation}
For  the holomorphic coefficient functions $t_{\bnu}(\bz_{\mathfrak{J}})$ on $\bbU^{N_0}_r$ defined in~\eqref{U_r_E},  Lemma~\ref{lem:uniform-head-moments} with order $\tau$ provides
compactly supported extensions $t_{e}^{\bnu}\in C^{\lceil\tau\rceil+N_0+1}(\R^{N_0};\mathbb V_{\C})$ satisfying
\[
M_s(t_{e}^{\bnu})\le D_1(D_2q)^{s+N_0} \|t_{\bnu}\|_{L^\infty(\bbU^{N_0}_r;\mathbb V_{\C})} \leq    D_1(D_2 q)^{s+N_b} S_{\bnu},
\qquad 0\le s\le\tau,
\]
where $\delta:=r/4$, $q:=\lceil\tau\rceil+N_b+1$, $D_1:=2[4(1+\delta)]^{N_b}$ and $D_2:=N_b/\delta$.
Here we have enlarged the bounds using
$N_0\le N_b$, $\lceil\tau\rceil+N_0+1\le q$,
and $D_2q\ge1$.
For the tail monomials, Lemma~\ref{lem:uniform-tail-moments} gives
integrable extensions $g_e^{\bnu} : \R^{|\mathfrak{J}^c|} \to \R$ satisfying
\[
M_s(g_e^{\bnu})
\le B_2B_3^s(\tau+4)^{s+1}B_1^{\|\bnu\|_1},
\qquad 0\le s\le\tau.
\]
If $\mathfrak{J}^c=\varnothing$, the expansion contains only
$t_{\bm0}=u_N$. Take $g_e^{\bm0}=1$ on $\R^0$,
with $\|g_e^{\bm0}\|_{L^1}=M_0(g_e^{\bm0})=1$
and $M_s(g_e^{\bm0})=0$ for $s>0$.
Then $u_{N,e}=t_e^{\bm0}$, and the estimates below remain valid.

\noindent\textbf{Step 2. Product moments.}
Set $H_{e}^{\bnu}(\by_{\mathfrak{J}},\by_{\mathfrak{J}^c})
:=t_{e}^{\bnu}(\by_{\mathfrak{J}})g_e^{\bnu}(\by_{\mathfrak{J}^c})$.
The moment estimates in the proof of
Lemma~\ref{lem:compatible-product} yield
\begin{equation}\label{eq:proof-product-moment-bound}
\begin{aligned}
&M_0(H_{e}^{\bnu})+M_\tau(H_{e}^{\bnu})
\le M_0(t_e^{\bnu})M_0(g_e^{\bnu}) + 2^{(\tau-1)_+} \bigl[M_\tau(t_e^{\bnu})M_0(g_e^{\bnu})+M_0(t_e^{\bnu})M_\tau(g_e^{\bnu})\bigr] \\
\le & D_1B_2S_{\bnu}B_1^{\|\bnu\|_1}
(D_2q)^{N_b}(\tau+4)  \Bigl[1+2^{(\tau-1)_{+}}(D_2 q)^\tau
+2^{(\tau-1)_{+}}[B_3(\tau+4)]^\tau\Bigr]\\
\le & C_*S_{\bnu}B_1^{\|\bnu\|_1}
(\tau+N_b+4)^{N_b + 1}\bigl(D_* (\tau+N_b+4) \bigr)^\tau,
\end{aligned}
\end{equation}
where $C_*:=2D_1B_2D_2^{N_b}$ and $D_*:=2\max\{D_2,B_3,1\}$.

\noindent\textbf{Step 3. Summation.}
Using the cutoff construction for $t_e^{\bnu}$
underlying \eqref{eq:he_L1} and the estimate for
$g_e^{\bnu}$ in \eqref{eq:polarization-monomial-L1},
we obtain
\[
\|t_{e}^{\bnu}\|_{L^1(\R^{N_0};\mathbb V_{\C})}
\le[2(1+\delta)]^{N_b}S_{\bnu},\qquad
\|g_e^{\bnu}\|_{L^1(\R^{|\mathfrak{J}^c|})}
\le \|g_e^{\bm0}\|_{L^1(\R^{|\mathfrak{J}^c|})} B_1^{\|\bnu\|_1}.
\]
The first bound follows from the cutoff construction,
$N_0\le N_b$, and
$\|t_{\bnu}\|_{L^\infty(\bbU^{N_0}_r;\mathbb V_{\C})}\le S_{\bnu}$.

Since $H_e^{\bnu}$ is a product in disjoint variables,
these estimates and \eqref{eq:tail-taylor-coeff-bound} give
\begin{equation}
\begin{split}
\sum_{\bnu\in\N_0^{\mathfrak{J}^c}}\|H_{e}^{\bnu}\|_{L^1(\R^N;\mathbb V_{\C})}
&\le \sum_{\bnu\in\N_0^{\mathfrak{J}^c}} [2(1+\delta)]^{N_b}\|g_e^{\bm0}\|_{L^1(\R^{|\mathfrak{J}^c|})} S_{\bnu} B_1^{\|\bnu\|_1} \\
&\le[2(1+\delta)]^{N_b}\| g_e^{\bm0}\|_{L^1(\R^{|\mathfrak{J}^c|})} S_*<\infty.
\end{split}
\end{equation}
Hence the series defining  $u_{N,e}:=\sum_{\bnu}H_{e}^{\bnu}$ converges absolutely in $L^1(\R^{N};\mathbb V_{\C})$.  On $U_N$, each summand agrees with the corresponding
Taylor term, so \eqref{eq:partial-taylor-expansion}
gives $u_{N,e}=u_N$. 
Termwise Fourier transformation, the triangle inequality,
and Tonelli's theorem, together with
\eqref{eq:proof-product-moment-bound} and
\eqref{eq:tail-taylor-coeff-bound}, yield
\begin{equation}\label{eq:proof-total-moment-bound}
\begin{split}
M_0(u_{N,e})+M_\tau(u_{N,e}) &\leq 
\sum_{\bnu\in\N_0^{\mathfrak{J}^c}} [M_0(H_{e}^{\bnu})+M_\tau(H_{e}^{\bnu})]  \\
&\le C_*S_{*}(\tau+N_b+4)^{N_b + 1}\bigl(D_* (\tau+N_b+4) \bigr)^\tau <\infty.
\end{split}
\end{equation}
Thus $u_{N,e}$ is an admissible extension.
By \eqref{eq:Fourier_moment_bound}, we obtain
\begin{equation} 
\begin{split}
 \|u_N\|_{\mathbb B_\tau(U_N;\mathbb V)}
\le M_0(u_{N,e})+M_\tau(u_{N,e}) \le  C_{V}[D_{B}(\tau+4)]^\tau,
\end{split}
\end{equation}
where  $D_{B}
:=eD_*(N_b+1)$ and $ C_{V}
:=\max\left\{1,C_*S_*(N_b+4)^{N_b+1}\right\}$.
Here we used  $(\tau+N_b+4)^{N_b+1}
\le (N_b+4)^{N_b+1}
e^\tau$ ,
and $\tau+N_b+4\le(N_b+1)(\tau+4)$.

\noindent\textbf{Step 4. Operator bound.}
The estimate is uniform in $N$. Applying the embedding
$\mathbb V\hookrightarrow \mathbb H$ to admissible extensions and taking
infima, followed by the supremum over $N$, gives
\[
\begin{aligned}
\|\mathcal G^\dagger\|_{\mathbb B_\tau(\mathbb B^{\bb};\mathbb H)}
&\le C_S
\|\mathcal G^\dagger\|_{\mathbb B_\tau(\mathbb B^{\bb};\mathbb V)} =C_S\sup_{N\ge1}\|u_N\|_{\mathbb B_\tau(U_N;\mathbb V)}
\le C_B[D_{B}(\tau+4)]^\tau,
\end{aligned}
\]
where we used $u_N=\mathcal G^\dagger_N$ and set
$ C_B:=\max\{1,C_S C_{V}\}$.
This proves \eqref{eq:ridge-map-Gamma}.

\end{proof}

%% file: Chapters/4_Error.tex

\section{Approximation error estimates}
\label{sec:Error}

In this section, we derive approximation error estimates for an FNO-type spectral neural operator approximating the coefficient-to-solution operator $\mathcal{G}^{\dag}$.
Without loss of generality, in the following discussion, we assume in addition that $b_j \neq 0$ for all $j\in\mathbb N_0$;
otherwise, zero coefficients may be discarded.
In addition, we assume that both $\mathbb{V}$ and $\mathbb{H}$ are spaces of scalar-valued functions; 
the extension to vector spaces follows from exactly the similar argument.

Our construction combines the input parameter extraction, the neural approximation of the truncated finite-dimensional coefficient-to-solution map, 
and the spatial reconstruction at the final stage, which also decomposes the total error into three components. 
Finally, Barron regularity of arbitrary order enables us
to optimize the order with respect to network width,
yielding a nearly exponential convergence rate.

To show how the operator Barron regularity provides the structure needed for quantitative approximation bounds, for simplicity, 
we work on the periodic cell $\Omega := Q_1 =(-1/2,1/2)^d$ and then choose the input basis $\{\psi_j\}_{j\ge0}$ from \eqref{eq_basis}, which satisfies
\[
\langle\psi_j,\psi_m\rangle
:=
\int_{\Omega}\psi_j(x)\psi_m(x)\,\dd x
=
\widetilde{\nu}_j\delta_{jm},
\qquad
\widetilde{\nu}_j:=
\begin{cases}
1,&j=0,\\
\frac12,&j\ne0.
\end{cases}
\]
This periodic trigonometric system is only used for the input representation,
while the solution space $\mathbb V$, including its boundary conditions, remains as specified above.

Define $\mathcal{F}$ and $\mathcal{F}^{-1}$ as the classical Fourier transform and its inverse on this periodic domain $Q_L$, respectively, i.e.,
$
\left(\mathcal{F} v \right)(\xi)
=\int_{\Omega}v(x)e^{-2\pi i\xi\cdot x}\,\dd x,
$ for $\xi\in\mathbb Z^d$.
We also write $\widehat{v}(\xi) = \mathcal{F}(v)(\xi)$
for the periodic Fourier coefficients.
Given a set of Fourier coefficients $\{\widehat{v}(\xi)\}_{\xi \in \mathbb Z^d}$, the inverse Fourier transform $\mathcal{F}^{-1}$ is defined as
\((\mathcal F^{-1} \widehat{v})(x)
=\sum_{\xi\in\mathbb Z^d}
\widehat{v}(\xi) e^{2\pi i\xi\cdot x}.\)

Following~\cite{li2021fourierneuraloperatorparametric}, we consider an FNO-type $\mathcal N:\mathbb B^\bb(\Omega)\to \mathbb V$ for approximating the coefficient-to-solution operator $\mathcal G^\dagger$. It lifts the scalar input $g$ to a feature field, applies a sequence of nonlinear transformations, and projects the resulting features onto a scalar output. For a prescribed depth $\depth\in\mathbb N$, this construction takes the form
$
\Ncal(g)
=
\bigl(\mathcal{}
\mathcal{Q}\circ\mathcal{L}_{\depth}\circ\mathcal{L}_{\depth-1}
\circ\cdots\circ\mathcal{L}_1\circ\mathcal{R}
\bigr)(g).
$
The lifting map acts pointwise as
\[
\mathcal{R}(g)(x)=R g(x) + \bm{\beta}_R,
\qquad
R,\bm{\beta}_R\in\mathbb{R}^{d_1}.
\]
For a $d_{\depth+1}$-channel feature field $v$, the output projection is defined pointwise by
\[
\mathcal{Q}(v)(x)=Q v(x) + \beta_Q,
\qquad
Q\in\mathbb{R}^{1\times d_{\depth+1}},
\quad \beta_Q\in\mathbb{R}.
\]
For each $j=1,\ldots,\depth$, the proposed FNO block reads
\begin{equation}
\label{eq:FourierLayerMatrix}
\mathcal{L}_j(v)(x)
=
W_j^{(2)}
\sigma_{j}\left(
W_j^{(1)}v(x)+\bm{\beta}_j(x)
+
\mathcal{F}^{-1}\bigl(
P_j(\xi) \cdot \mathcal{F}(v)(\xi)
\bigr)(x)
\right).
\end{equation}
The matrices $W_j^{(1)}\in\mathbb{R}^{\tilde{d}_j\times d_j}$, $W_j^{(2)} \in\mathbb{R}^{d_{j+1}\times \tilde{d}_j} $ and the bias function $\bm{\beta}_j(x) \in \mathbb{R}^{\tilde{d}_j}$ parameterize the pointwise affine operations.
The matrix-valued Fourier multiplier $P_j:\mathbb Z^d\to\mathbb C^{\tilde d_j\times d_j}$ has finite support and satisfies $P_j(-\xi)=\overline{P_j(\xi)}$,
so that the Fourier term is real valued for real-valued inputs.
In particular, if $P_j(\xi)=0$ for $\xi\ne\bm0$, the Fourier term is independent of $x$.

\subsection{Network construction}

The proposed network has an encoder-decoder structure.
Recall $g = a - \bar{a}\in \mathbb{B}^{\bb}$, $g_N = a_N - \bar{a} \in \mathbb{B}^{\bb}_N$,
where $\mathbb{B}^{\bb}_N$ is the parametric set obtained by retaining the first $N$ modes in $\mathbb{B}^{\bb}$.
Define the input encoder and decoder by
\[
\mathcal E_N:\mathbb B^{\mathbf b}\to U_N,
~~~
\mathcal E_N\bigl(g(\cdot;\mathbf y)\bigr)=\mathbf y_N;
\qquad
\mathcal D_N: U_N \to \mathbb B^{\bb}_N,
~~~
\mathcal D_N(\mathbf y_N)= g_N.
\]
In particular, $\mathcal{D}_N = \Phib\circ\mathcal P_N^\top$.
For the solution space $\mathbb V$, recall $\mathbb V_M:=\operatorname{span}\{e_1,\ldots,e_M\}\subset \mathbb V$, define the encoder and decoder by
\[
\begin{aligned}
\mathcal E_M:\mathbb V\to\mathbb R^M,
\qquad
\mathcal E_M(v)
&:=\bigl(\langle v,e_1\rangle_{\mathbb H},\cdots,\langle v,e_M\rangle_{\mathbb H}\bigr)^\top,\\
\mathcal D_M:\mathbb R^M\to \mathbb V_M,
\qquad
\mathcal D_M(\mathbf c)
&:=\sum_{m=1}^M c_m e_m.
\end{aligned}
\]
{We distinguish the input and output encoders and decoders by the subscripts $N$ and $M$.}
In addition, let $\Pi_N:\mathbb B^{\bb}\to\mathbb B^{\bb}_N$ be the input truncation operator and $g_N=\Pi_N g$. 
Again, this is distinguished from $\Pi_M$ in \eqref{eq_VM_proj} by subscripts to keep notation simple.

{The remaining component maps the retained input parameters to the output coefficients in \eqref{eq_Gdecomp} and \eqref{eq:solution-spectral-expansion}}.
To facilitate the discussion, we define
\begin{equation}
\label{F_definition}
\mathbf{F}_{N,M}:U_N\to\mathbb R^M, \qquad
\mathbf{F}_{N,M}(\mathbf y_N)
:=
\bigl(F_{1,N}(\mathbf y_N),\ldots,F_{M,N}(\mathbf y_N)\bigr)^\top.
\end{equation}
Combining these maps and using \eqref{eq_operator_decomp}, 
we obtain the decomposition underlying the network construction:
\begin{equation}
\label{eq_decompo}
\mathcal D_M\circ\mathbf{F}_{N,M}\circ\mathcal E_N
=
\Pi_M\circ\mathcal G^{\dagger}_N\circ\mathcal E_N
=
\Pi_M\circ\mathcal G^\dagger\circ\Phib\circ\mathcal P_N^{\top}\circ\mathcal E_N.
\end{equation}
{We next construct FNO-type networks that realize $\mathcal E_N$ and $\mathcal D_M$ exactly and approximate the coefficient map $\mathbf F_{N,M}$. The construction of feature extraction layers is similar to \cite[Lemma~7]{2021Kovachki}, 
{but errors from numerical evaluation of the Fourier coefficients are not included here.}}
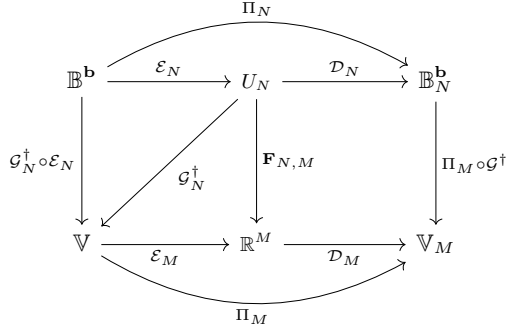
\begin{figure}[htbp]
\centering
\begin{tikzcd}[row sep=5em, column sep=5em]
\mathbb B^{\bb}
  \arrow[r, "\mathcal E_N"]
  \arrow[d,
    "\mathcal G_N^\dagger\circ\mathcal E_N"']
  \arrow[rr, bend left=30, "\Pi_N"]
& \scalebox{0.85}{$U_N$}
  \arrow[r, "\mathcal{D}_N"]
  \arrow[d, "\mathbf{F}_{N,M}"]
  \arrow[dl, "\mathcal G_N^\dagger", pos=0.48]
& \mathbb B^{\bb}_N
  \arrow[d, "\Pi_M\circ\mathcal G^\dagger"]
\\
\mathbb V
  \arrow[r, "\mathcal E_M"']
  \arrow[rr, bend right=30, "\Pi_M"']
& \scalebox{0.85}{$\mathbb R^M$}
  \arrow[r, "\mathcal D_M"']
& \mathbb V_M
\end{tikzcd}
\caption{Encoder-decoder decomposition of the solution operator.}
\label{fig:encoder-decoder-decomposition}
\end{figure}

\begin{lemma}
\label{lem:exact-fno-parameter-extraction}
There exists an FNO-type network \(\widetilde{\mathcal E}_N = \mathcal{L}_3 \circ \mathcal{L}_2 \circ \mathcal{L}_1 \circ \mathcal{R}: \mathbb B^{\bb}\to U_N\) 
with width \(4 N\), 
such that, for every $g(x;\mathbf y)=\sum_{j\in\mathbb N_0}b_jy_j\psi_j(x) $, 
\begin{equation*}
\widetilde{\mathcal E}_N\left( g \right)
=\mathcal E_{N}(g)
{=\mathcal P_N \by = \by_N,}
\quad \by_N \in U_N.
\end{equation*}
\end{lemma}

\begin{proof}
To describe the construction,
we introduce the element-wise product \(\odot\) and the Kronecker product \(\otimes\). We use the following notation throughout the proof.
\begin{equation}
\label{eq:s1s2}
\begin{alignedat}{2}
I_N
&= \text{size-$N$ identity matrix},
&\qquad \mathbf{1}_N
&= \text{size-$N$ column vector of ones}, \\
\mathbf{s}_1
&:= (1,-1,1,-1)^\top \in \mathbb{R}^4,
&\qquad \mathbf{s}_2
&:= (1,-1,-1,1)^\top \in \mathbb{R}^4, \\
\Gamma
&:= \operatorname{diag}
\bigl(\widetilde{\nu}_0 b_0,\ldots,
\widetilde{\nu}_{N-1} b_{N-1}\bigr),
&\qquad \boldsymbol{\psi}
&:= (\psi_0,\ldots,\psi_{N-1})^\top.
\end{alignedat}
\end{equation}

For any $g(x;\mathbf y) \in \mathbb B^{\bb}(\Om)$, the goal is to recover
\(\mathbf y_N=(y_0,\ldots,y_{N-1})^\top\).
For \(0\leq i<N\), the orthogonality of trigonometric functions yields
\begin{equation}
\label{eq:extractor}
\langle g,\psi_i\rangle_{L^2(\Omega)}
=\sum_{j\ge0}b_jy_j\langle\psi_j,\psi_{i}\rangle
=b_{i}\widetilde{\nu}_{i}y_{i}.  
\end{equation}
Thus, the extraction formula \eqref{eq:extractor} can be written in matrix form $\langle g,\boldsymbol\psi\rangle=\Gamma \mathbf y_N$. Since \(b_i \neq 0\) and \(\widetilde{\nu}_i > 0\), \(\Gamma^{-1}\) exists. The next step is to implement these inner products $ \ip{g}{\psi_i}$ and the multiplication by \(\Gamma^{-1}\) using Fourier layers and activation functions.
The layer parameters are listed in Table \ref{tab:exact-parameter-extraction}. 
{All activations act componentwise, and $\operatorname{ReLU}^2(t):=(\max\{t,0\})^2$.}
\begin{table}[htbp]
\centering
\caption{Layer parameters of the exact parameter extractor in Lemma~\ref{lem:exact-fno-parameter-extraction}. 
}
\label{tab:exact-parameter-extraction}
\begingroup
\small
\setlength{\tabcolsep}{4pt}
\renewcommand{\arraystretch}{1.4}
\begin{tabular}{@{}cccccc@{}}
\toprule
Map & $R$ or $W_j^{(1)}$ & Bias & $P_{j}$ & Activation & $W_j^{(2)}$ \\
\midrule
$\mathcal R$ & $\mathbf s_1\otimes\mathbf1_N$ & $0_{4N}$ & --- & --- & --- \\
$\mathcal L_1$ & $0_{4N\times4N}$ & $\mathbf s_2\otimes\boldsymbol\psi(x)$ & $\mathbf1_{\mathbb K_N}(\xi)I_{4N}$ & $\operatorname{ReLU}^2$ & $I_{4N}$ \\
$\mathcal L_2$ & $0_{4N\times4N}$ & $\mathbf s_2\otimes\mathbf1_N$ &
$
\mathbf{1}_{\{\mathbf0\}} (\xi) \left[ (\mathbf s_1\odot\mathbf s_2)\otimes D \right]
$ & $\operatorname{ReLU}^2$ & $I_{4N}$ \\
$\mathcal L_3$ & $I_{4N}$ & $0_{4N}$ & $0_{4N\times4N}$ & $\operatorname{ReLU}$ & $T$ \\
\bottomrule
\end{tabular}
\par\smallskip
\begin{minipage}{\linewidth}
\end{minipage}
\endgroup
\end{table}

We now verify the construction.
The lifting layer $\mathcal{R}$ sends $g(x;\by)$ to $\mathbf{v}_0(x):=\mathcal{R}(g)(x)=\mathbf s_1\otimes(g(x;\by)\mathbf1_N) \in \R^{4N}$. For the first Fourier layer $\mathcal{L}_1$, we define
\(\mathbb K_N=\bigcup_{j=0}^{N-1} \{\mathbf k(j),-\mathbf k(j)\}\) and \(\widetilde g_N(x) =\sum_{\xi\in \mathbb K_N}\widehat g(\xi)
e^{2\pi i \ip{\xi}{x}}\).
{Since the multiplier $P_1$ retains the Fourier modes in
$\mathbb K_N$, we obtain}
\begin{equation}
\mathcal{F}^{-1}[P_1\widehat{\mathbf{v}}_0](x)
=
\sum_{\xi \in \mathbb K_N} \widehat{\mathbf{v}}_0(\xi) e^{2 \pi i \ip{\xi}{x}}
=
\mathbf s_1\otimes\bigl(\widetilde{g}_N(x)\mathbf1_N\bigr).
\end{equation}
The output of the first layer is
\(
\mathbf{v}_1(x):= \mathcal{L}_1(\mathbf{v}_0)(x) = \sigma_1\!\bigl(
\mathbf s_1\otimes(\widetilde g_N(x)\mathbf1_N)
+\mathbf s_2\otimes\boldsymbol\psi(x) \bigr) \in \R^{4N} \). 
The task of the second Fourier layer $\mathcal{L}_2$ is to implement the inner product by preserving the zero Fourier mode. 
Set 
\(D := \frac14
\begin{pmatrix}I_N&I_N&-I_N&-I_N\end{pmatrix}\) 
{in the multiplier $P_2$}. From the identity \(\sigma_1(t)+\sigma_1(-t)=t^2\) for $\textrm{ReLU}^2$, we can obtain
\(
D \mathbf{v}_1(x)
=\tfrac14\!\left[
(\widetilde g_N+\boldsymbol\psi)^2
-(\widetilde g_N-\boldsymbol\psi)^2\right]
=\widetilde g_N\boldsymbol\psi
\). 
Since the Fourier support of each $\psi_i$, $0\le i<N$, is contained in $\mathbb K_N$, Fourier orthogonality gives
\(
\langle \widetilde{g}_N,\boldsymbol{\psi}\rangle=\langle g,\boldsymbol{\psi}\rangle
\). 
By \(
\widehat{\mathbf{v}}_1(0)
=\int_{\Omega}\mathbf{v}_1(x)\,dx
\), we can obtain
\begin{equation}
D\widehat{\mathbf{v}}_1(0)
=\langle\widetilde g_N,\boldsymbol\psi\rangle
=\langle g,\boldsymbol\psi\rangle
=\Gamma\mathbf y_N.
\end{equation}
We obtain
\(
\mathbf{v}_2(x):= \mathcal{L}_2(\mathbf{v}_1)(x)=\sigma_2\!\left(
(\mathbf s_1\odot\mathbf s_2)\otimes(\Gamma\mathbf y_N)
+\mathbf s_2\otimes\mathbf1_N
\right) 
\) and \(\mathbf{v}_2\) is a spatially constant field with nonnegative components. 
Since \(\sigma_3=\operatorname{ReLU}\) and $\mathbf{v}_2$ has nonnegative components,
\(
\sigma_3(\mathbf{v}_2)=\mathbf{v}_2.
\)
The final linear matrix is
\(
T:=
\frac14
\begin{pmatrix}
\Gamma^{-1}&-\Gamma^{-1}&
\Gamma^{-1}&-\Gamma^{-1}
\end{pmatrix}
\) 
satisfying
\begin{equation}
\begin{aligned}
T\mathbf{v}_2
&=\tfrac14\Gamma^{-1}
\big[(\Gamma\mathbf y_N+\mathbf1_N)\odot (\Gamma\mathbf y_N+\mathbf1_N)-(\Gamma\mathbf y_N-\mathbf1_N)\odot (\Gamma\mathbf y_N-\mathbf1_N)\big] \\
&=\Gamma^{-1}(\Gamma\mathbf y_N)=\mathbf y_N.
\end{aligned}
\end{equation}
Thus, \(\mathbf{v}_3 := \mathcal{L}_3(\mathbf{v}_2) = T\mathbf{v}_2 = \by_N\). Let
$\widetilde{\mathcal E}_N:=\mathcal L_3\circ\mathcal L_2\circ\mathcal L_1\circ\mathcal R$.
Identifying the spatially constant output field with its value in $\mathbb R^N$, we obtain $\widetilde{\mathcal E}_N(g)=\mathcal E_N(g)=\mathbf y_N$ exactly.
\end{proof}

Next, we construct the fourth Fourier layer $\mathcal{L}_4$ to  approximate $\mathbf{F}_{N,M}(\by_N)$. 

\begin{lemma}
\label{lem:spectral-coefficient-approximation}
Let \(N\ge1\) be an integer, let \(\tau>0\) and suppose that \(u_N\in\mathbb B_\tau(U_N;\mathbb H)\). Let \(\{e_m\}_{m\ge1}\) be the fixed \(\mathbb H\)-orthonormal basis.
For all integers \(M, K\ge1\), there are scalar cosine networks \(\NN_{m,K}\), each with at most \(K\) hidden units, such that \(\mathbf{N}_K:=(\NN_{1,K},\ldots,\NN_{M,K})^\top\) satisfies
\begin{equation}
\label{eq:NNError}
\|\mathbf{F}_{N,M}-\mathbf{N}_K\|_{L^2(U_N,\rho_N;\mathbb R^M)}
\le C_\tau K^{-\left(\frac12+\frac{\tau}{N}\right)}
\|u_N\|_{\mathbb B_\tau(U_N;\mathbb H)},
\end{equation}
where $\mathbf{F}_{N,M}$ is defined in \eqref{F_definition}, and $C_{\tau}$ is given in \eqref{eq:explicit-BarronApprox-constant}.
\end{lemma}

\begin{proof}
The coordinate estimate \eqref{eq:barron-coordinate-bound} gives
\(
\sum_{m=1}^{M}
\|F_{m,N}\|_{\mathbb B_\tau(U_N)}^2
\le \|u_N\|_{\mathbb B_\tau(U_N;\mathbb H)}^2.
\)
Each \(F_{m,N}\) belongs to \(\mathbb B_\tau(U_N)\).
By  Lemma~\ref{lem:explicit-BarronApprox}, for every \(m\le M\), there is a cosine network
\[
\NN_{m,K}(\mathbf y_N)
=\mathbf{c}_m^\top\sigma(A_m\mathbf y_N+\bm{\theta}_m),
\qquad
A_m\in\mathbb R^{K\times N},\quad
\mathbf{c}_m, \bm{\theta}_m\in\mathbb R^K,
\quad \sigma(t)=\cos t,
\]
satisfying
\[
\|F_{m,N}-\NN_{m,K}\|_{L^2(U_N;\rho_N)}
\le C_\tau K^{-\frac{1}{2}-\frac{\tau}{N}}
\|F_{m,N}\|_{\mathbb B_\tau(U_N)}.
\]
Networks using fewer than \(K\) units are padded with zero output weights.
Then,
\[
\begin{aligned}
\|\mathbf{F}_{N,M}-\mathbf{N}_K\|_{L^2(U_N,\rho_N;\mathbb R^M)}^2
=&\sum_{m=1}^{M}
\|F_{m,N}-\NN_{m,K}\|_{L^2(U_N,\rho_N)}^2\\
\le & C_\tau^2K^{-1- \frac{2\tau}{N}} \sum_{m=1}^{M}\|F_{m,N}\|_{\mathbb B_\tau(U_N)}^2
\le C_\tau^2K^{-1- \frac{2\tau}{N}}
\|u_N\|_{\mathbb B_\tau(U_N;\mathbb H)}^2.
\end{aligned}
\]
Taking square roots proves the error estimate \eqref{eq:NNError}.

To realize \(\mathbf{N}_K\), set
\[
A=\begin{pmatrix}A_1\\ \vdots\\ A_M\end{pmatrix}
\in\mathbb R^{MK\times N},\quad
\bm{\theta}=\begin{pmatrix}\bm{\theta}_1\\ \vdots\\ \bm{\theta}_M\end{pmatrix}
\in\mathbb R^{MK},\quad
A_{\mathrm{out}}=
\begin{pmatrix}
\mathbf{c}_1^\top&\cdots&\bm 0\\
\vdots&\ddots&\vdots\\
\bm 0&\cdots&\mathbf{c}_M^\top
\end{pmatrix}
\in\mathbb R^{M\times MK}.
\]
The \(m\)-th component of \(A_{\mathrm{out}}\sigma(A\mathbf y_N+\bm \theta)\) is \(\NN_{m,K}(\mathbf y_N)\). Hence, 
\(
A_{\mathrm{out}}\sigma(A\mathbf y_N+\bm \theta)=\mathbf{N}_K(\mathbf y_N).
\)

For the fourth Fourier layer $\mathcal L_4$, 
choose \(W_4^{(1)}=A\), \(\boldsymbol\beta_4=\bm \theta\), \(W_4^{(2)}=A_{\mathrm{out}}\), {$\sigma_4(t)=\cos t$}, and \(P_4(\xi)=0_{MK\times N}\) for every \(\xi\). 
For the constant field \(\mathbf{v}_3(x) \equiv \mathbf y_N\), the identity
\(
\mathcal L_4(\mathbf{v}_3)(x)
=A_{\mathrm{out}}\sigma_4(A\mathbf y_N+\bm \theta)
=\mathbf{N}_K(\mathbf y_N)
\) holds for any $x \in \Om$.
\end{proof}

Finally, we realize the spatial reconstruction map
$\mathcal D_M$ exactly using fixed basis features.

\begin{lemma}
\label{lem:exact-spatial-reconstruction-layers}
Let $M\ge1$ be an integer. 
There exist two Fourier layers $\mathcal L_5$ and
$\mathcal L_6$, with zero Fourier multipliers and
maximum channel width $4M$, together with a linear
output projection $\mathcal Q$, such that
\(
\widetilde{\mathcal D}_M
:=\mathcal Q\circ\mathcal L_6\circ\mathcal L_5
\)
satisfies
\[
(\widetilde{\mathcal D}_M\mathbf c)(x)
=(\mathcal D_M\mathbf c)(x)
=\sum_{m=1}^M c_m e_m(x),
\qquad \mathbf c\in\mathbb R^M,
\quad\text{for a.e. }x\in\Omega.
\]
\end{lemma}

\begin{proof}
{Set $E_M(x):=(e_1(x),\ldots,e_M(x))^\top$, where $e_1,\ldots,e_M$ are the prescribed output basis functions.} 
We treat \(\mathbf c \in \R^{M}\) as a spatially constant field. 
We retain the notation from \eqref{eq:s1s2}.
Then
\( (\mathcal D_M\mathbf c)(x)=\sum_{m=1}^M c_me_m(x)
=\mathbf1_M^\top\bigl(\mathbf c\odot E_M\bigr)\) indicates that the network needs to calculate \(c_m e_m(x)\) component by component first, and then sum the \(M\) products. 

With the parameters in Table \ref{tab:exact-spatial-reconstruction}, the fifth layer $\mathcal{L}_5$ produces 
$$\mathbf v_5(x):=\mathcal L_5(\mathbf c)(x)
=\sigma_5\!\left(\mathbf s_1\otimes\mathbf c
+\mathbf s_2\otimes E_M(x)\right)
\in\mathbb R^{4M}.$$ 
As every component of \(\mathbf v_5\) is nonnegative, the ReLU in the sixth layer $\mathcal{L}_6$ acts as the identity.
{Set $S_M:=\frac14
\begin{pmatrix}I_M&I_M&-I_M&-I_M\end{pmatrix}
\in\mathbb R^{M\times4M}.$
Using $\sigma_5(t)+\sigma_5(-t)=t^2$, we obtain}
\[
\begin{aligned}
S_M \mathbf v_5
&=\frac14\left[
(\mathbf c+E_M(x))\odot(\mathbf c+E_M(x))
-(\mathbf c-E_M(x))\odot(\mathbf c-E_M(x))
\right] 
=\mathbf c\odot E_M(x).
\end{aligned}
\]
{Thus, $\mathbf v_6(x) 
:=\mathcal L_6(\mathbf v_5)(x)
=S_M\mathbf v_5(x)
=\mathbf c\odot E_M(x)$. 
The output projection therefore gives}
\(
\mathcal Q(\mathbf v_6)(x)=\mathbf1_M^\top \mathbf v_6(x)
=\sum_{m=1}^M c_me_m(x)
=(\mathcal D_M\mathbf c)(x)
\) for almost every \(x\in\Omega\).
\end{proof}

\begin{table}[htbp]
\centering
\caption{Layer parameters of the exact spatial reconstruction in Lemma~\ref{lem:exact-spatial-reconstruction-layers}. 
}
\label{tab:exact-spatial-reconstruction}
\begingroup
\small
\setlength{\tabcolsep}{4pt}
\renewcommand{\arraystretch}{1.4}
\begin{tabular}{@{}cccccc@{}}
\toprule
Map & $W_j^{(1)}$ or $Q$ & Bias & $P_{j}$ & Activation & $W_j^{(2)}$ \\
\midrule
$\mathcal L_5$ & $\mathbf s_1\otimes I_M$ & $\mathbf s_2\otimes E_M(x)$ & $0_{4M\times M}$ & $\operatorname{ReLU}^2$ & $I_{4M}$ \\
$\mathcal L_6$ & $I_{4M}$ & $0_{4M}$ & $0_{4M\times4M}$ & $\operatorname{ReLU}$ & $S_M$ \\
$\mathcal Q$ & $\mathbf1_M^\top$ & $0$ & --- & --- & --- \\
\bottomrule
\end{tabular}
\par\smallskip
\begin{minipage}{\linewidth}
\end{minipage}
\endgroup
\end{table}

The corresponding FNO-type spectral neural operator is then defined by
\(
\Ncal
:=
\widetilde{\mathcal D}_M
\circ
\mathbf{N}_{K}
\circ
\widetilde{\mathcal E}_N = \mathcal Q\circ\mathcal L_6 \circ \cdots \circ \mathcal{L}_1 \circ \mathcal{R}.
\)
Equivalently, for \(g(\cdot;\by) \in \mathbb B^{\bb}\),
\[
[\Ncal \left(g(\cdot; \by)\right)](x)
=
\sum_{m=1}^M
\NN_{m,K}(\by_N)e_m(x).
\]

{In the next two subsections, we use the preceding network
construction to prove
Theorems~\ref{thm:fno-approximation-error-with-complexity}
and~\ref{thm:fno-approximation-error-with-complexity2}.}


\subsection{Proof of Theorem \ref{thm:fno-approximation-error-with-complexity}}
\label{subsec:thm_error1}
{Recall that $\Ncal$ is the network constructed above.}
Fix $n\ge N$ and write $g_n=g_n(\cdot;\by_n)$ and
$g_N=g_N(\cdot;\by_N)$, where $\by_n\in U_n$ and
$\by_N=(y_0,\ldots,y_{N-1})$.
The triangle inequality yields
\begin{equation}\label{eq:error-decomposition}
\begin{aligned}
&\norm{\mathcal{G}^{\dag}(g_n)-\Ncal(g_n)}_{L^2(U_n,\rho_n;\mathbb H)} \\
&\leq
\underbrace{
\norm{\mathcal{G}^{\dag}(g_n)-\mathcal{G}^{\dag}(g_{N})}_{L^2(U_n,\rho_n;\mathbb H)}
}_{\mathrm{I}}
+
\underbrace{
\norm{\mathcal{G}^{\dag}(g_{N})-\Pi_M\mathcal{G}^{\dag}(g_{N})}_{L^2(U_N,\rho_N;\mathbb H)}
}_{\mathrm{II}}
\\
&\quad+
\underbrace{
\norm{\Pi_M\mathcal{G}^{\dag}(g_{N})-\Ncal(g_n)}_{L^2(U_N,\rho_N;\mathbb H)}
}_{\mathrm{III}}.
\end{aligned}
\end{equation}
{The functions in $\mathrm{II}$ and $\mathrm{III}$ depend
only on $\by_N$, since the network extracts only the first
$N$ parameters. Integrating out the remaining coordinates
with respect to the product probability measure gives
the displayed $L^2(U_N,\rho_N;\mathbb H)$ norms.}

We then estimate each term individually.
By the definition in
\eqref{eq:fno-total-error-with-complexity-high-1},
$\mathrm{I} \le \varepsilon_N^{\mathrm{in}}$ and $\mathrm{II} =\varepsilon_{N,M}^{\mathrm{out}}$ are both immediate.
%
Finally for $\mathrm{III}$, using exact parameter extraction and spatial reconstruction in Lemmas \ref{lem:exact-fno-parameter-extraction} and \ref{lem:exact-spatial-reconstruction-layers},
the decomposition in \eqref{eq_decompo},
and the $\mathbb H$-orthonormality of
$\{e_m\}$, we obtain
\[
\begin{aligned}
\mathrm{III}
&=\norm{ \mathcal{D}_M \circ \mathbf{F}_{N,M} \circ \mathcal{E}_N(g_n) - \widetilde{\mathcal D}_M
\circ
\mathbf{N}_{K}
\circ
\widetilde{\mathcal E}_N(g_n) }_{L^2(U_N,\rho_N;\mathbb H)} \\
&=\|\mathbf F_{N,M}-\mathbf N_K\|
_{L^2(U_N,\rho_N;\mathbb R^M)}.  
\end{aligned}
\]
Lemma~\ref{lem:spectral-coefficient-approximation}
and the definition of the operator Barron norm
in \eqref{eq_operatornorm} therefore give
\[
\begin{aligned}
\mathrm{III}
&\le C_\tau K^{-(\frac12+\frac{\tau}{N})}
\|u_N\|_{\mathbb B_\tau(U_N;\mathbb H)}\le C_\tau K^{-(\frac12+\frac{\tau}{N})}
\|\mathcal G^\dagger\|
_{\mathbb B_\tau(\mathbb B^\bb;\mathbb H)}.
\end{aligned}
\]
The constructed network has maximum channel width
$\max\{4N,MK,4M\}$ and is independent of $n$.
Combining the three estimates, substituting
\eqref{eq:explicit-BarronApprox-constant}, and taking
the supremum over $n\ge N$ proves
\eqref{eq:fno-total-error-with-complexity-high}.


\subsection{Proof of Theorem~\ref{thm:fno-approximation-error-with-complexity2}}
\label{subsec:error_2}

{Using \eqref{eq:ridge-map-Gamma},
\eqref{eq:explicit-BarronApprox-constant}, 
and $\max\{8\pi e,2(\tau+1)\}\le4\pi e(\tau+2)$, we bound the neural approximation term in
\eqref{eq:fno-total-error-with-complexity-high} by}
\begin{equation*}
\begin{aligned}
R_{N,K}(\tau)
&:=C_\tau C_B [ D_{B} (\tau+4)]^\tau
K^{-\frac12-\frac\tau N}
\le4 C_B K^{-1/2}
\left[ \widetilde{D}_B (\tau+2)(\tau+4)K^{-1/N}\right]^\tau,
\end{aligned}
\end{equation*}
where $\widetilde{D}_B:=4\pi eD_{B}$.
We choose $\tau$ to minimize the dominant factor
$\left[\widetilde{D}_B\tau^2K^{-1/N}\right]^\tau$.
Its logarithm is
$\Phi_K(\tau):=\tau(\log \widetilde{D}_B+2\log\tau-\frac{\log K}{N})$, and
\[
\Phi_K'(\tau)=\log \widetilde{D}_B+2\log\tau+2-\frac{\log K}{N},
\qquad
\Phi_K''(\tau)=\frac2\tau>0.
\]
Thus, the minimum is achieved at
$\tau=
\frac{K^{1/(2N)}}{e\sqrt{\widetilde{D}_B}} =:\tau_K$. Consequently, $\widetilde{D}_B\tau_K^2K^{-1/N}=e^{-2}$. Using $\log(1+x)\le x$ for $x\ge0$, we obtain
\[
\begin{aligned}
\left[\widetilde{D}_B(\tau_K+2)(\tau_K+4)K^{-1/N}\right]^{\tau_K}
=e^{-2\tau_K}
\left(1+\frac2{\tau_K}\right)^{\tau_K}
\left(1+\frac4{\tau_K}\right)^{\tau_K}
\le e^6e^{-2\tau_K}.
\end{aligned}
\]
Therefore, with $C_{\mathcal{N}}:= 4e^6 C_B$, we obtain
\begin{equation*}
R_{N,K}(\tau_K)
\le4e^6 C_B K^{-1/2}e^{-2\tau_K}
= C_{\mathcal{N}} K^{-1/2}
\exp\left(-\frac{K^{1/(2N)}}{e\sqrt{\pi eD_{B}}}\right).
\end{equation*}
Combining this bound with the two truncation terms in \eqref{eq:fno-total-error-with-complexity-high} proves \eqref{eq:fno-optimized-error}. 

%% file: Chapters/5_Experiment.tex
\section{Numerical Experiments}
In this section, we examine the performance of FNO on the approximation of operators $\mathcal{G}^{\dagger}$ presented in the Theorem \ref{thm:fno-approximation-error-with-complexity} and Theorem \ref{thm:fno-approximation-error-with-complexity2}.
To guarantee an over-parameterized regime~\cite{2023KontolatiGoswamiShieldsEtAl} such that the approximation/representation error is dominant instead of the optimization error~\cite{2023SiegelHongJinEtAl}, we consider a simple 1D 
parametric elliptic problem \eqref{eq_case1_darcyflow} on $\Omega=(0,1)$ with the source $f=1$. 
For $\by=(y_j)_{j\in\mathbb N_0}$, let $u(\cdot;\by)\in \mathbb V$ be the solution corresponding to the coefficient family
\begin{equation}
\label{eq:exp-darcy-coefficient}
\begin{aligned}
a(x;\by)&=2+g(x;\by),\\
g(x;\by)&=\sum_{j=0}^{\infty}b_jy_j\psi_j(x),~~
b_j=C_{\sigma}(j+1)^{-r},~~
\psi_j(x)=\cos(2\pi jx),
\end{aligned}
\end{equation}
where 
$C_\sigma>0$ and the parameters $y_j$ are independent and uniformly distributed on $[-1,1]$. 
Since $|\psi_j|\le1$, 
to ensure the uniform lower bound of $a$ and the coercivity,
$C_\sigma$ needs to be reasonably small such that 
$C_\sigma\sum_{j=1}^{\infty} j^{-r}<2$. 
Then, the decay rate $r>1$  specifies the spatial regularity of the input family. 
For $0<p<r-1$, choose $w_j=1+j^p$, which yields
\(
\|\bb\|_{\ell_w}
=C_\sigma\sum_{j=0}^{\infty}\frac{1+j^p}{(j+1)^r}<\infty.
\)




We train an FNO to approximate the coefficient-to-solution
operator $\mathcal G^\dagger:g\mapsto u$.
For given $N$ and $M$, the training data consist of truncated
inputs $g_N$ and projected outputs $\Pi_Mu$.
At test time, we evaluate the network on $g$ and measure
the error against the corresponding solution $u$.
By Lemma \ref{lem:input-truncation}, the input truncation quantity in the error estimate satisfies
\(
\varepsilon_N^{\mathrm{in}}
=O(N^{-r}).
\)
In addition, we use the fixed $\mathbb H$-orthonormal basis $e_m(x)=\sqrt2\sin(m\pi x)$ to reconstruct the final solution. 
Since the one-dimensional elliptic solution $u_N$ is $H^2$-regular, its projection onto the first $M$ modes satisfies
$$
 \|(I-\Pi_M)u_N\|_{L^2(\Omega)}\le \|u_N''\|_{L^2(\Omega)}/[\pi^2(M+1)^2] .
$$ 
so the output truncation error is $\varepsilon_{N,M}^{\mathrm{out}} = O(M^{-2})$.
We consider coefficient decay exponents $r\in\{1.5,2,3\}$.
We use an FNO that has 4 layers and 16 Fourier modes per layer;
see other hyperparameters of the network in the caption of Figure~\ref{fig:standard-fno-parameter-study}.
It is trained on 8192 samples for 600 epochs by Adam with learning rate $2\times 10^{-3}$ and batch size 512.
The results are reported in Figure~\ref{fig:standard-fno-parameter-study} which show the approximation error of FNO as a function of the input truncation dimension $N$, the number of retained output modes $M$, and the network width $K$. The vertical axis reports the $L^2$ test error on a logarithmic scale.
The empirical slopes $s$ reported in the legends are estimated by least-squares linear regression in log--log coordinates:
\(
\log E(t)\approx s\log t+c,
\)
where $t\in\{N,M,K\}$ denotes the parameter varied in each panel and $E(t)$ is the corresponding test error. The fitted slopes characterize the empirical dependence of the error on each parameter over the displayed range.
The observed convergence rates with respect to \(N\) and \(M\) are broadly consistent with the theoretical predictions, although the rates in \(N\) for \(r=1.5\) and \(r=2\) are slightly faster than predicted.
The curves level off once the other error contributions become dominant.
For convergence with respect to \(K\), the observed rate is slightly faster than \(K^{-1/2}\), as expected from the additional exponential factor.

\begin{figure}[htbp]
\centering
\includegraphics[width=0.32\linewidth]{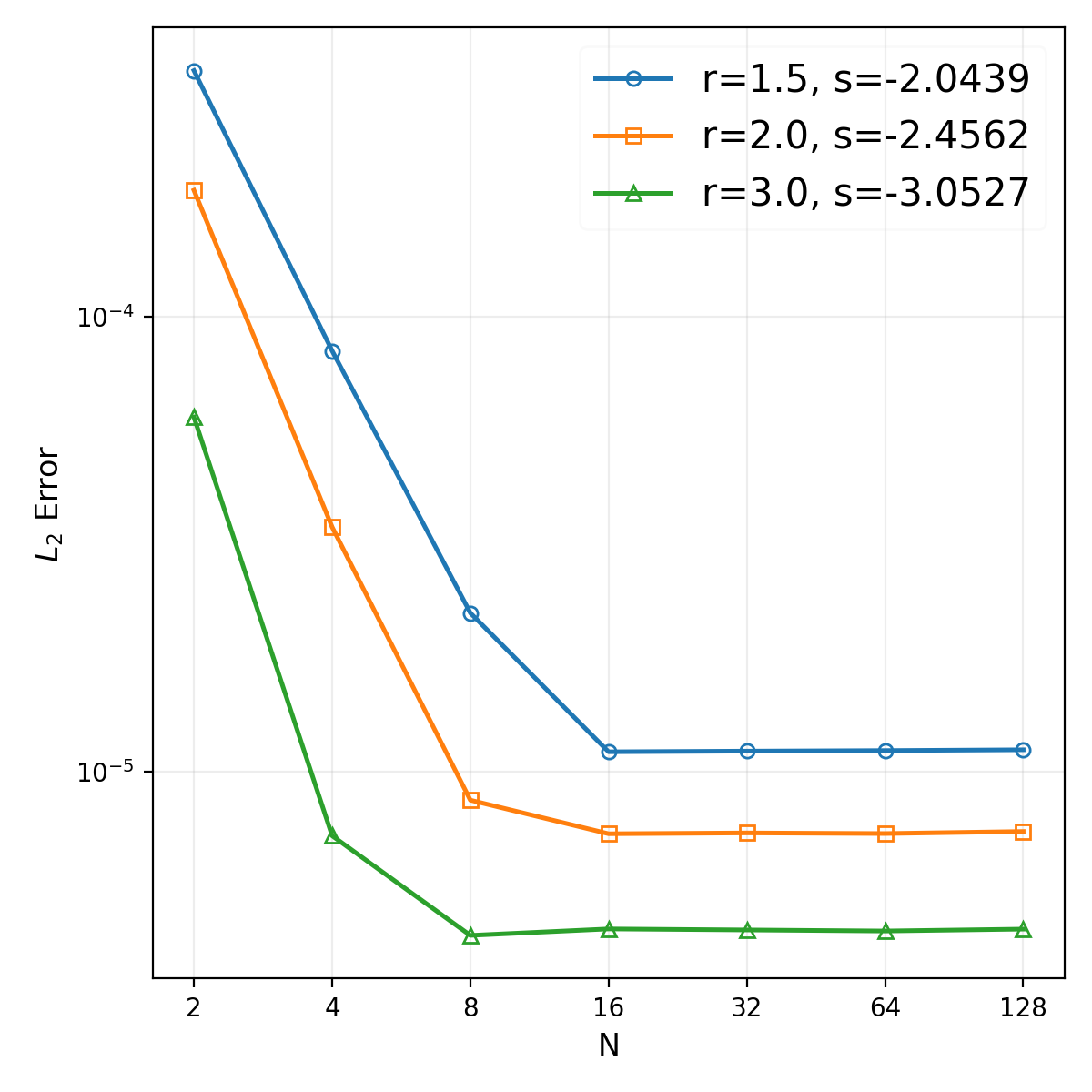}\hfill%
    \includegraphics[width=0.32\linewidth]{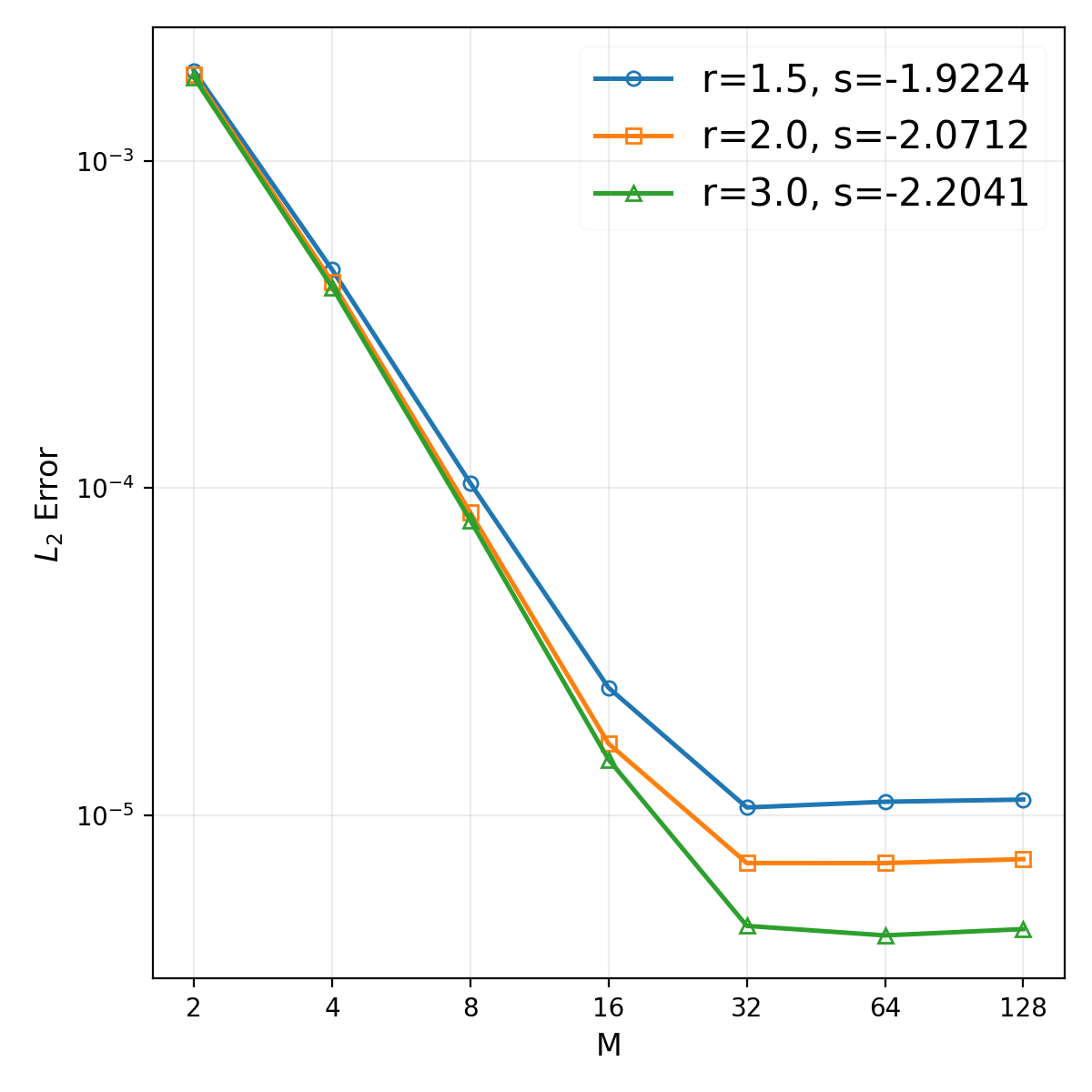}\hfill%
    \includegraphics[width=0.32\linewidth]{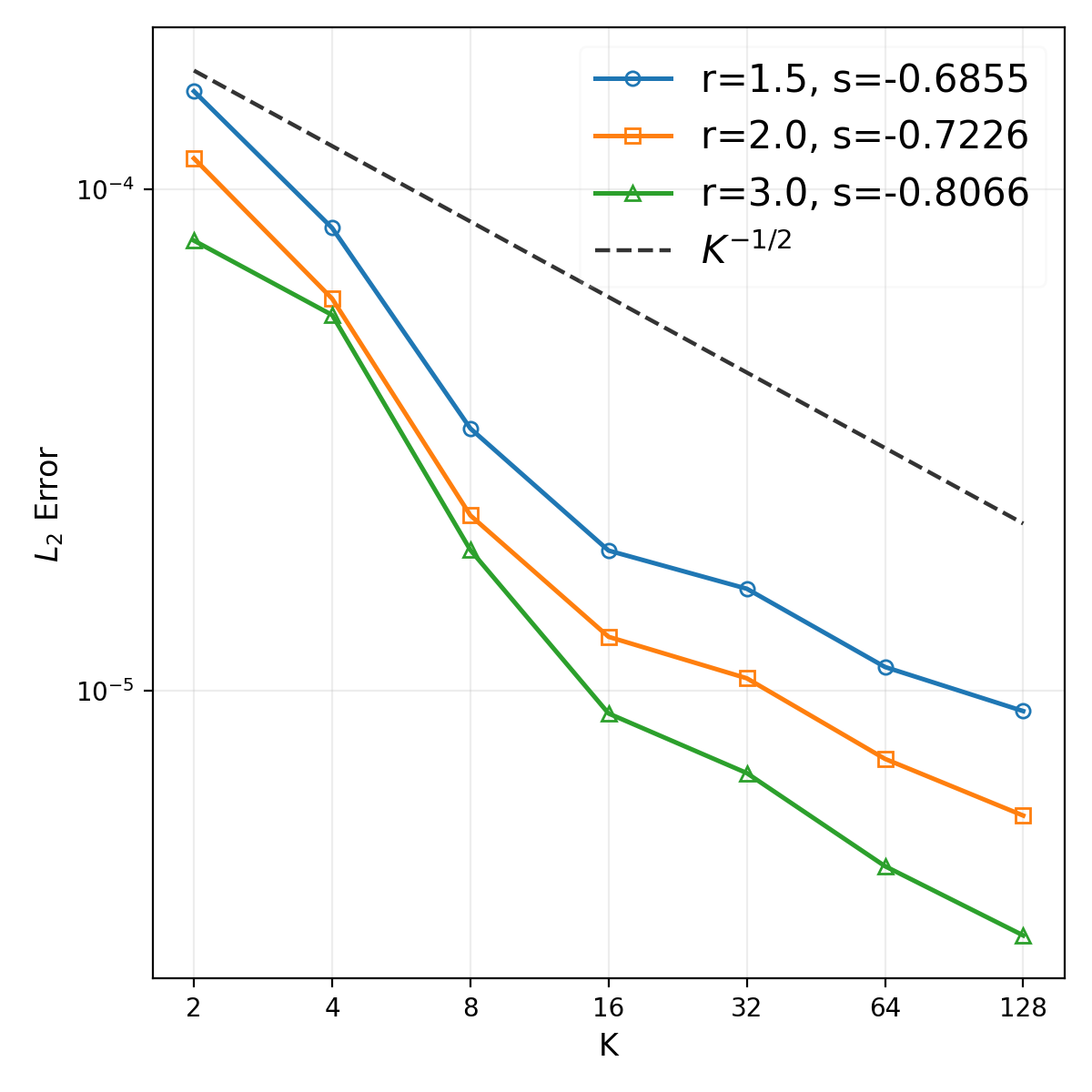}
\caption{$L^2$ test errors of FNO, evaluated on $8192$ test pairs for coefficient decay exponents $r\in\{1.5,2,3\}$.
\textbf{Left}: with $M=128$ and $K=64$ fixed, the error v.s. $N$. 
\textbf{Middle}: with $N=64$ and $K=64$ fixed, the error v.s. $M$.
\textbf{Right}: with $N=64$ and $M=128$ fixed, the error v.s. $K$.
}
\label{fig:standard-fno-parameter-study}
\end{figure}

Next, we examine the predicted stretched-exponential convergence with respect to $K$.
To make the dependence on $K^{1/(2N)}$ in \eqref{eq:fno-optimized-error} visible over the tested range of $K$, 
we generate the input coefficient data by retaining only the first one or two terms in \eqref{eq:exp-darcy-coefficient}, 
corresponding to $N=1$ and $N=2$, respectively. 
In this case, we use 2 Fourier layers with 16 modes per layer, fixing $M=\mathrm{32}$. The training sample sizes for $N=1$ and $N=2$ are 100 and 400, respectively.
For each $N$, we train and test the network on this
finite-parameter family, with target
$\mathcal G^\dagger(g_N)=u_N$.
Thus, no input-truncation error arises relative to this target.
Figure~\ref{fig:standard-fno-parameter-study2} presents the numerical results.
For each $N$, we fit the errors to
$AK^{-1/2}\exp\bigl(-cK^{1/(2N)}\bigr)$,
where $A$ and $c$ are determined from the data and reported in the figure.
The observed decay is consistent with the predicted stretched-exponential rate.

\begin{figure}[htbp]
\centering
\includegraphics[width=0.32\linewidth]{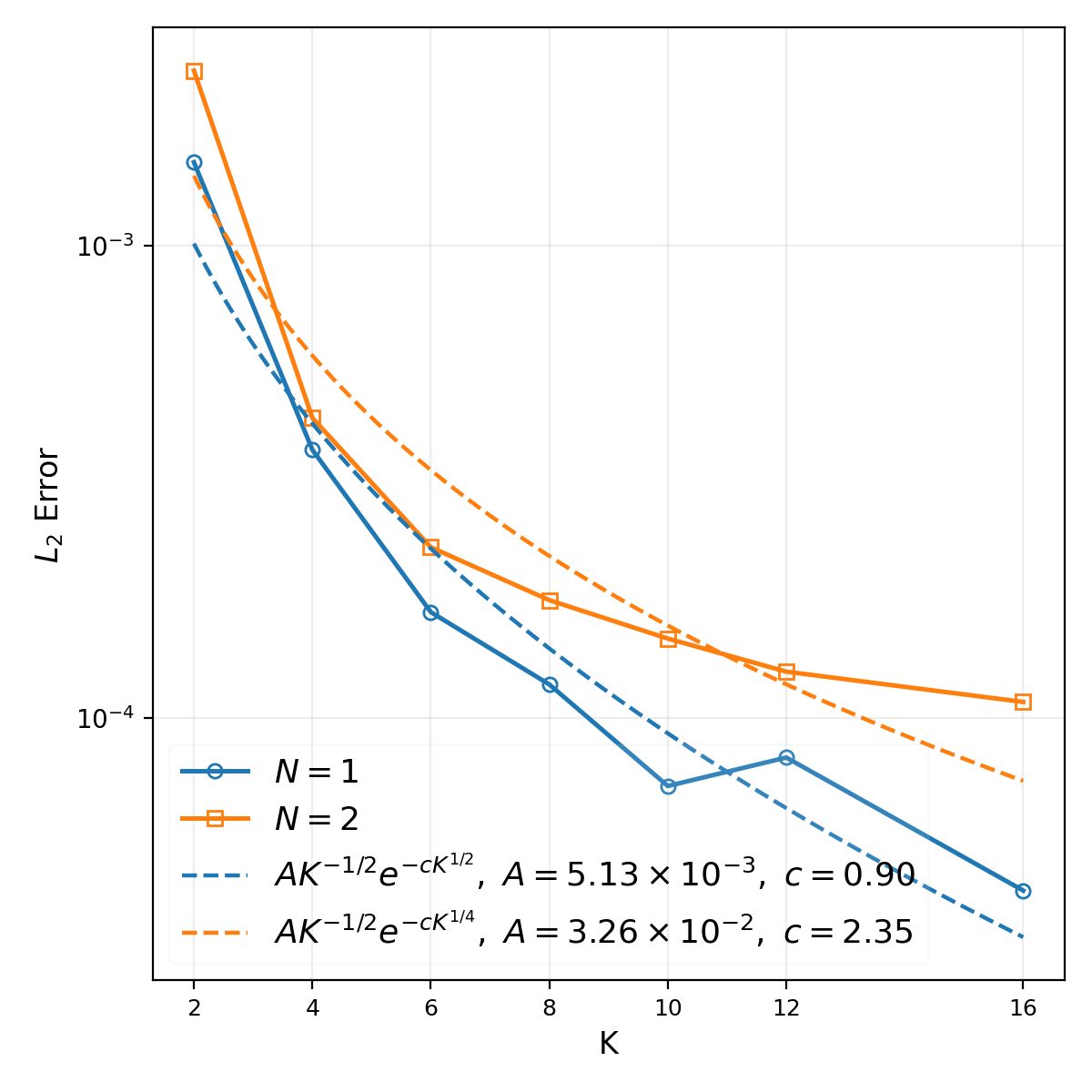}
\caption{$L^2$ test errors of FNO as functions of the network width $K$ for coefficient fields generated with $r=2$ and input dimensions $N=1$ and $N=2$. The dashed curves show the fitted rates $AK^{-1/2}\exp(-cK^{1/(2N)})$.}
\label{fig:standard-fno-parameter-study2}
\end{figure}

%% file: Chapters/6_Summary.tex
\section{Conclusion}
We introduce a spectral Barron regularity framework for nonlinear operators between infinite-dimensional function spaces. 
Under this framework, using holomorphic techniques and harmonic analysis, 
we show that the elliptic operators admit arbitrary Barron regularity with explicit constants in the norm bounds.
This regularity is then translated into approximation theory for an explicit FNO-type encoder-decoder network.
By optimizing the Barron regularity order, we show that the approximation rate of such an operator learner can be nearly exponential with respect to network width.
The Darcy-flow experiments illustrate the predicted roles of input regularity, truncation dimensions, and network width.
Together, these results provide a direct quantitative link between elliptic PDE structure, operator Barron regularity, and neural-operator approximation.
In future work, we will apply this framework to other PDE
solution operators for operator learning.


\section*{Acknowledgments}
Parts of the analysis and settings are refined by GPT-6 with online search and access to references, 
including the removal of unnecessary assumptions from the initial analysis and locating some existing results in the literature. 
The authors subsequently verified and simplified the resulting arguments and organized them into those results in sections of harmonic analysis and holomorphic extension. 
GPT-6 was also used to improve the language and presentation of the manuscript and to assist in checking the numerical study. The authors take full responsibility for the final content.

%% file: Chapters/7_Appendix_v2.tex

\section{Representation of Barron spaces}
\label{append_barron_represent}

\begin{lemma}
\label{lem_mul_barron}
Let $p>0$. For every $f,g\in\mathbb B_p(\R^d)$, we have $fg\in\mathbb B_p(\R^d)$ and
\begin{equation}
\label{lem_mul_barron_eq0}
\|fg\|_{\mathbb B_p(\R^d)}
\leq 2^{(p-1)_+}
\|f\|_{\mathbb B_p(\R^d)}
\|g\|_{\mathbb B_p(\R^d)}.
\end{equation}
\end{lemma}
\begin{proof}
Set $W_p(\xi):=1+\|\xi\|_1^p$.
Using $(a+b)^p\leq 2^{(p-1)_+}(a^p+b^p)$ for $a,b\geq0$ and the triangle inequality, we obtain
\begin{equation}
\label{lem_mul_barron_eq2}
\begin{aligned}
W_p(\xi)
\leq 1+2^{(p-1)_+}
\left(\|\xi-\eta\|_1^p+\|\eta\|_1^p\right)
\leq 2^{(p-1)_+}W_p(\xi-\eta)W_p(\eta),
\qquad \xi,\eta\in\R^d.
\end{aligned}
\end{equation}

Since $\widehat f,\widehat g\in L^1(\R^d)$, Fourier inversion gives bounded continuous representatives of $f$ and $g$.
Together with $f,g\in L^1(\R^d)$, this implies $fg\in L^1(\R^d)$.
Thus, under the Fourier convention fixed in Section~\ref{sec:barron-spaces}, the product formula reads
\[
\widehat{fg}(\xi)
=(\widehat f*\widehat g)(\xi)
=\int_{\R^d}\widehat f(\xi-\eta)\widehat g(\eta)\,d\eta.
\]
In particular, Tonelli's theorem and \eqref{lem_mul_barron_eq2} yield
\begin{align*}
\|fg\|_{\mathbb B_p(\R^d)}
&\leq\int_{\R^d}\int_{\R^d}
W_p(\xi)|\widehat f(\xi-\eta)|\,|\widehat g(\eta)|\,d\eta\,d\xi\\
&\leq 2^{(p-1)_+}
\int_{\R^d}\int_{\R^d}
W_p(\xi-\eta)|\widehat f(\xi-\eta)|
W_p(\eta)|\widehat g(\eta)|\,d\eta\,d\xi\\
&=2^{(p-1)_+}
\|f\|_{\mathbb B_p(\R^d)}
\|g\|_{\mathbb B_p(\R^d)},
\end{align*}
where the last equality follows from the change of variables $\zeta=\xi-\eta$.
This proves \eqref{lem_mul_barron_eq0}.
\end{proof}

\begin{lemma}
\label{lem:lem_phi1}
Let $p>0$ and choose $L>0$ such that
$\overline{\Omega}\subset Q_L:=(-L/2,L/2)^d$.
Let $\{\psi_j\}_{j\in\N_0}$ be the real trigonometric
basis on $Q_L$ defined in \eqref{eq_basis}, and set
$w_j:=1+\|\bk(j)/L\|_1^p$.
Fix a cutoff function $\chi$ satisfying \eqref{eq_cutoff_fun}.
Then the mapping
$\Phi:\ell_w\to\mathbb B_p(\Omega)$,
$\balpha\mapsto
\left.\left(\sum_{j\in\N_0}\alpha_j\psi_j\right)\right|_\Omega$
is well defined, linear, and bounded.
For every $\balpha\in\ell_w$, the defining series converges
absolutely and uniformly on $Q_L$, and
\begin{equation}
\label{eq:lem_phi1}
\|\Phi(\balpha)\|_{\mathbb B_p(\Omega)}
\le C_{\Omega,p}\|\balpha\|_{\ell_w},
\qquad
C_{\Omega,p}:=
2^{(p-1)_+}\|\chi\|_{\mathbb B_p(\R^d)}.
\end{equation}
\end{lemma}

\begin{proof}
Let $\balpha\in\ell_w$. Since $w_j\geq1$ and $\norm{\psi_j}_{L^{\infty}}=1$, we have
\(
\sum_{j\in\N_0}\norm{\alpha_j\psi_j}_{L^{\infty}(Q_L)}
\le \sum_{j\in\N_0}w_j|\alpha_j|
=\|\balpha\|_{\ell_w}<\infty.
\)
Thus, the series $H_{\balpha}:=\sum_{j\in\N_0}\alpha_j\psi_j$ converges absolutely and uniformly on $Q_L$.

To estimate the Barron norm, 
define $F_{\balpha}:=\chi H_{\balpha}$, extended by zero outside $Q_L$.
Then $F_{\balpha}\in L^1(\R^d)$ and $F_{\balpha}|_\Omega=\Phi(\balpha)$.
Set $\xi_j:=\bk(j)/L$ and $W_p(\xi):=1+\|\xi\|_1^p$.
The definition of the Fourier transform gives
\[
\begin{aligned}
\widehat{F}_{\balpha}(\xi)
&=
\int_{\R^d}
\chi(x)\sum_{j\in\mathbb N_0}\alpha_j\psi_j(x)
e^{-2\pi i\langle x,\xi\rangle}\,\dd x
=
\sum_{j\in\mathbb N_0}
\alpha_j\,\widehat{\chi\psi_j}(\xi).
\end{aligned}
\]
For the real trigonometric basis, there holds
\[
\widehat{\chi\psi_j}(\xi)
=
\begin{cases}
\widehat{\chi}(\xi),
& \psi_0(x)=1,
\\[1mm]
\dfrac12\left[
\widehat{\chi}(\xi - \xi_j)+\widehat{\chi}(\xi + \xi_j)
\right],
& \psi_j(x)=\cos\bigl(2\pi\langle x, \xi_j\rangle\bigr),
\\[2mm]
\dfrac1{2i}\left[
\widehat{\chi}(\xi-\xi_j)-\widehat{\chi}(\xi+\xi_j)
\right],
& \psi_j(x)=\sin\bigl(2\pi\langle x,\xi_j\rangle\bigr).
\end{cases}
\]
Here, the interchange of summation and integration follows from the $L^1$-integrability property.
Therefore, we have
\[
\left|\widehat{F}_{\balpha}(\xi)\right|
\le
\frac12\sum_{j\in\mathbb N_0}|\alpha_j|
\left[
\left|\widehat{\chi}(\xi - \xi_j)\right|
+
\left|\widehat{\chi}(\xi + \xi_j)\right|
\right].
\]
Moreover, the weight inequality in Lemma~\ref{lem_mul_barron} and the identity $W_p(\xi_j)=w_j$ give
\[
W_p(\xi)
\le 2^{(p-1)_+} W_p(\xi \pm \xi_j)W_p(\xi_j)
= 2^{(p-1)_+} w_j W_p(\xi \pm \xi_j).
\]
Using Tonelli's theorem and the change of variables $\bm{\zeta}=\xi \pm \xi_j$, we obtain
\begin{align*}
\int_{\R^d}
W_p(\xi)|\widehat{F}_{\balpha}(\xi)|\,d\xi
&\le
\frac{1}{2}\sum_{j\in\N_0}|\alpha_j|
\int_{\R^d}
W_p(\xi) \left( |\widehat\chi(\xi-\xi_j)| + |\widehat\chi(\xi+\xi_j)| \right) \,\dd\xi
\\
&\le
2^{(p-1)_+}\sum_{j\in\N_0}w_j|\alpha_j|
\int_{\R^d}
W_p(\bm{\zeta})|\widehat\chi(\bm{\zeta})| \, d\bm{\zeta}
=2^{(p-1)_+}\|\chi\|_{\mathbb B_p(\R^d)}
\|\balpha\|_{\ell_w}.
\end{align*}
By Lemma~\ref{lem_cut_Barrom}, $\|\chi\|_{\mathbb B_p(\R^d)}<\infty$, so $F_{\balpha}$ is an admissible extension of $\Phi(\balpha)$ in $\mathbb B_p(\R^d)$.
The definition of the Barron norm then yields \eqref{eq:lem_phi1}. 
\end{proof}


\section{Barron regularity of cutoff functions}

Recall that a neighborhood of $\bar{\Omega}$ is contained in $Q_L$ and define a cutoff function:
\begin{equation}
\label{eq_cutoff_fun}
\chi\in C_c^\infty(Q_L), ~~~ \text{with} ~ \chi=1 ~ \text{on a neighborhood of }\overline\Omega.
\end{equation}
In this section, we discuss the Barron regularity of such cutoff functions. 
A one-dimensional cutoff construction and quantitative derivative estimates are given in Lemma \ref{lem_cutoff}. 

\begin{lemma}
\label{lem_cut_Barrom}
Let $p>0$, $Q_L = (-L/2, L/2)^d \supseteq \overline{\Omega}$,
and $\chi$ satisfy \eqref{eq_cutoff_fun}.
Then $\chi\in \mathbb{B}_p(\R^d)$.
\end{lemma}
\begin{proof}
Since $\chi\in C_c^\infty(\R^d)$, its Fourier
transform decays faster than any polynomial, as follows
from repeated integration by parts. Consequently,
\begin{equation}
\label{eq_CchiM}
C_{\chi,m}
:=
\sup_{\zeta\in\R^d}
(1+\| \zeta \|_2)^m|\widehat\chi(\zeta)| < \infty , ~~~ \forall m>0.
\end{equation}
Moreover, notice that
$
\| \zeta \|_1 \le \sqrt{d} \| \zeta \|_2,
$
and thus
\begin{equation}
\label{eq:cutoff-first-bound}
\begin{split}
\|\chi\|_{\mathbb B_{p}(\R^d)}
=
\int_{\R^d}
\left(1+\|\zeta\|_1^p\right)
|\widehat\chi(\zeta)|\,\dd \zeta \le
C_{\chi,m}
\int_{\R^d}
\left(1 + d^{p/2} \| \zeta \|_2^p\right)
(1+ \| \zeta \|_2)^{-m}\,\dd \zeta.
\end{split}
\end{equation}
Introduce
$\Gamma(s):=
\int_{0}^{\infty} t^{s-1}e^{-t}\,\dd t$, $s>0$, as the Euler gamma function.
Choosing $m>d+p$, by polar integration and the beta-integral identity, we have
\begin{equation}
\begin{aligned}
\|\chi\|_{\mathbb B_{p}(\R^d)}
\le
C_{\chi,m}
\frac{2\pi^{d/2}}
{\Gamma(d/2)\Gamma(m)}
\left[
\Gamma(d)\Gamma(m-d)
+
d^{p/2}
\Gamma(d+p)\Gamma(m-d-p)
\right].
\end{aligned}
\end{equation}
\end{proof}

\begin{lemma}
\label{lem_cutoff}
Let $q\in\mathbb{N}$ and $\delta > 0$. 
Then, there exists $\chi_{q,\delta} \in C^q(\R)$ satisfying
\begin{equation}
\label{eq_chi_R}
    0\leq \chi_{q,\delta}\leq 1,
    \qquad
    \chi_{q,\delta}\equiv 1\ \text{on }[-1,1],
    \qquad
    \operatorname{supp}(\chi_{q,\delta})\subset[- 1-\delta ,1+\delta ],
\end{equation}
and
\begin{equation}
\left\|
\chi_{q,\delta}^{(k)}
\right\|_{L^1(\R)}
\le
2\left[ \frac{2(q+1)}{\delta} \right]^{k-1},
\qquad
1\leq k\leq q.
\label{eq:cutoff-delta-L1}
\end{equation}
\end{lemma}
\begin{proof}
We employ the standard construction approach by integrating a scaled cardinal B-spline \cite{1999Unser}.
Set $\overline{q}=q+1$. 
We denote convolution by $*$ and define the functions

\[
b_{{q,\delta}}(t):=\frac{\overline{q}}{\delta}\mathbf{1}_{[0,\delta/\overline{q}]}(t),
~~~
\text{and} 
~~~
 \varphi_{q,\delta}(t):=  b_{q,\delta}^{*\overline{q}}:=\underbrace{b_{q,\delta}*\cdots*b_{q,\delta}}_{\overline{q}\text{ factors}}.
\]
The convolution identities imply
\[
\varphi_{q,\delta} \ge 0,
~~~~ \int_{\R} \varphi_{q,\delta} \dd t = 1,
~~~~ \text{supp}(\varphi_{q,\delta}) \subset [0,\delta].
\]
The associated transition function is then defined by $ \Theta_{q,\delta}(t):=\int_{-\infty}^{t}\varphi_{q,\delta}(s)\,ds$.
Now, we are ready to define the finite-order cutoff $\chi_{q,\delta}(t)$ by
\[
    \chi_{q,\delta}(t)=
    \begin{cases}
        1, & |t|\le 1,\\
        \Theta_{q,\delta}(1+\delta-|t|), & 1<|t|<1+\delta,\\
        0, & |t|\ge 1+\delta.
    \end{cases}
\]
This construction satisfies \eqref{eq_chi_R}.

Next, we estimate its derivatives.
We shall frequently use Young's inequality for convolution:
\begin{equation}
\label{eq_young}
\|f*g\|_{L^1(\R)}
\le \|f\|_{L^1(\R)}\|g\|_{L^1(\R)},
\qquad
\|\mu*f\|_{L^1(\R)}
\le \|\mu\|_{\mathrm{TV}}\|f\|_{L^1(\R)},
\end{equation}
where $f,g\in L^1(\R)$ and $\mu$ is a finite signed measure.
For every $1\leq k\leq q$, the derivative
$\chi_{q,\delta}^{(k)}$ vanishes on the intervals where
$\chi_{q,\delta}$ is constant.
Consequently, with a change of variable, we obtain
\begin{equation}
\label{eq_chi_theta}
\left\|
\chi_{q,\delta}^{(k)}
\right\|_{L^1(\R)}
= 
2\int_0^\delta
\left|
\Theta_{q,\delta}^{(k)}(s)
\right|\,\dd s =
2
\left\|
\Theta_{q,\delta}^{(k)}
\right\|_{L^1(\R)}
\end{equation}
where in the last equality, we have used that $\Theta_{q,\delta}$ is constant outside $[0,\delta]$. 
Now, it remains to estimate $\left\|\Theta_{q,\delta}^{(k)}\right\|_{L^1(\R)}$.
For every $2\leq k\leq q$, we write down
\begin{equation}
\label{eq_theta_derik}
    \Theta_{q,\delta}^{(k)}
    =
    (b'_{q,\delta})^{*(k-1)}*b_{q,\delta}^{*(\overline{q}+1-k)},
    ~~~
    \text{and}
    ~~~
    \Theta_{q,\delta}^{(1)} = \varphi_{q,\delta}.
\end{equation}
Repeated application of \eqref{eq_young}, together with
$\|b_{q,\delta}\|_{L^1(\R)}=1$, yields
$
\left\|b_{q,\delta}^{*l}\right\|_{L^1(\R)}
\le
\left\|b_{q,\delta}\right\|_{L^1(\R)}^{l}
=1$, for every integer $l\geq 1$.
Moreover, the distributional derivative
$b'_{q,\delta}
=\frac{\overline q}{\delta}
\bigl(\delta_0-\delta_{\delta/\overline q}\bigr)
$ is a finite signed measure satisfying
$\|b'_{q,\delta}\|_{\mathrm{TV}}
\le 2\overline q/\delta$.
Applying the measure version of \eqref{eq_young} repeatedly
to \eqref{eq_theta_derik} gives,
\[
    \bigl\|\Theta_{q,\delta}^{(k)}\bigr\|_{L^1(\R)}
    \le
    \|b'_{q,\delta}\|_{\textrm{TV}}^{k-1}
    \bigl\|b_{q,\delta}^{*(\overline{q}+1-k)}\bigr\|_{L^1(\R)}  
    \le (2\overline{q}/\delta)^{k-1}.
\]
Substituting this estimate into \eqref{eq_chi_theta} completes the proof.
\end{proof}

\section{Explicit Barron approximation by cosine networks}
\label{append_explicit_BarronApprox}

In this appendix, we prove the cosine approximation rate of
\cite[Theorem~1]{SiegelXu2022HighOrder} in the probability
$L^2$ norm used in the main text, with an explicit constant depending only
on the Barron regularity order.

General estimates for orthogonal projection onto polynomial spaces are
classical; see, for example,
\cite[Theorem~2.3]{CanutoQuarteroni1982}.  For the
exponential atoms used below, the following elementary estimate for a
total-degree polynomial space is sufficient.

\begin{lemma}[Projection of an exponential]
\label{lem:exponential-polynomial-projection}
Let $N,p\ge1$ be integers, set $U_N=[-1,1]^N$, and let $\rho_N$ be the uniform probability measure on $U_N$ defined in \eqref{eq_prob_measure}.
Define $\mathbb{P}_{p-1}:=\operatorname{span}
\{\by^{\bm\nu}:\bm\nu\in\mathbb N_0^N,\ \|\bm\nu\|_1\le p-1\}$,
and let $\mathcal{Q}_{p-1}$ be the $L^2(U_N,\rho_N)$-orthogonal projection onto
$\mathbb{P}_{p-1}$, extended complex linearly when necessary. 
Then, for every $\bm\xi\in\R^N$,
\begin{equation}
\label{eq:exponential-polynomial-projection}
\left\|(I- \mathcal{Q}_{p-1})e^{2\pi i\bm\xi\cdot(\,\cdot\,)}\right\|_{L^2(U_N,\rho_N)}
\le
\min\left\{1,\frac{(2\pi\|\bm\xi\|_1)^p}{p!}\right\}.
\end{equation}
The same estimate holds with the exponential replaced by
$\cos(2\pi\bm\xi\cdot(\,\cdot\,)+\theta)$ for any $\theta\in\R$.
\end{lemma}
\begin{proof}
The zero polynomial gives the bound $1$.  The Taylor polynomial
$q_{\bm\xi,p}(\by):=\sum_{\ell=0}^{p-1}
(2\pi i\bm\xi\cdot\by)^\ell/\ell!$ belongs to $\mathbb{P}_{p-1}$, and its
integral remainder, together with
$|\bm\xi\cdot\by|\le\|\bm\xi\|_1$ on $U_N$, gives
\[
\|e^{2\pi i\bm\xi\cdot(\,\cdot\,)}-q_{\bm\xi,p}\|_{L^\infty(U_N)}
\le \frac{(2\pi\|\bm\xi\|_1)^p}{p!}.
\]
The best-approximation property of $ \mathcal{Q}_{p-1}$ and $\rho_N(U_N)=1$ prove
\eqref{eq:exponential-polynomial-projection}.  Taking the real part after
multiplication by $e^{i\theta}$ proves the cosine statement.
\end{proof}

\begin{lemma}[Explicit cosine approximation]
\label{lem:explicit-BarronApprox}
Let $N\ge1$ be an integer, let $\tau>0$, and set
$U_N=[-1,1]^N$. Let $\rho_N$ denote the uniform probability
measure on $U_N$. For every real-valued $g\in\mathbb B_\tau(U_N)$ and every
integer $K\ge1$, there exists a real single-hidden-layer cosine network with width at most $K$:
\[
g_K(\by)=\sum_{j=1}^K
c_j\sigma(\boldsymbol\omega_j\cdot\by+\theta_j),
\qquad \sigma(t)=\cos t,
\qquad \by\in U_N,
\]
where $c_j,\theta_j\in\R$ and
$\boldsymbol\omega_j\in\R^{N}$, such that
\begin{equation}
\label{eq:explicit-BarronApprox-error}
\|g-g_K\|_{L^2(U_N,\rho_N)}
\le C_{\tau}
\|g\|_{\mathbb B_\tau(U_N)}
K^{-\frac12-\frac{\tau}{N}},
\end{equation}
where 
\begin{equation}
\label{eq:explicit-BarronApprox-constant}
C_{\tau}
=4\left[\max\{8\pi e,2(\tau+1)\}\right]^\tau.
\end{equation}
\end{lemma}

\begin{proof}
We first represent $g$ as an expectation of cosine functions.
Direct random sampling yields a cosine network with error
$\mathcal O(K^{-1/2})$.
To improve this rate, we estimate the component of the sampling
error orthogonal to a polynomial space and approximate its
polynomial projection by a second cosine network.
Combining the two networks yields the improved rate
$\mathcal O(K^{-1/2-\tau/N})$.
The case $g=0$ is immediate, so assume $g\ne0$.
For $K=1$, take $g_1=0$ and use
$\|g\|_{L^2(U_N,\rho_N)}
\le\|g\|_{L^\infty(U_N)}
\le\|g\|_{\mathbb B_\tau(U_N)}$
and $C_\tau\ge1$.
We henceforth assume $K\ge2$.

\textbf{Step 1: Probabilistic representation.}
By Definition~\ref{Def:SpatialNorm}, choose an admissible
extension $g_e$ of $g$ such that
\begin{equation}
\label{eq:explicit-BarronApprox-extension}
B_e:=\int_{\R^N}
(1+\|\bm\xi\|_1^\tau)|\widehat g_e(\bm\xi)|\,\dd\bm\xi
\le2\|g\|_{\mathbb B_\tau(U_N)}.
\end{equation}
Denote $v_\tau(\bm\xi):=1+\|\bm\xi\|_1^\tau$, and define
\[
\varpi(\bm\xi):=B_e^{-1}v_\tau(\bm\xi)|\widehat g_e(\bm\xi)|,
\qquad
\theta(\bm\xi):=\arg\widehat g_e(\bm\xi),
\qquad
a_{\bm\xi}(\by):=
\frac{\cos(2\pi\bm\xi\cdot\by+\theta(\bm\xi))}{v_\tau(\bm\xi)},
\]
where $\arg0:=0$. Then $\varpi$ is a probability density on $\R^N$.
Since $g$ is real valued, Fourier inversion yields
\begin{equation}
\label{eq:explicit-BarronApprox-representation}
\begin{aligned}
g(\by)
=\operatorname{Re}\int_{\R^N}
\widehat g_e(\bm\xi)e^{2\pi i\bm\xi\cdot\by}\,\dd\bm\xi =B_e\int_{\R^N}a_{\bm\xi}(\by)\varpi(\bm\xi)\,\dd\bm\xi,
\qquad \by\in U_N.
\end{aligned}
\end{equation}
Then, we have
\begin{equation}
\label{eq:explicit-BarronApprox-representation_1}
g(\by) = B_e \mathbb E_{\bm\xi} \left[ a_{\bm\xi}(\by) \right].
\end{equation}

Direct random sampling yields a cosine network with error
$\mathcal O(K^{-1/2})$. To improve this rate, we introduce a
polynomial correction to the sampling error.
Let $\mathbb P_{p-1}$ and $\mathcal Q_{p-1}$ be the polynomial
space and orthogonal projection defined in
Lemma~\ref{lem:exponential-polynomial-projection}, with
\begin{equation}
\label{p_KN}
p:=\left\lfloor(K/2)^{1/N}\right\rfloor.
\end{equation}
This choice reserves $p^N\le K/2$ hidden units for the
polynomial correction, balancing the polynomial degree
and the number of samples.
In the following, we distinguish the cases $p\ge\tau$ and $p<\tau$.

\textbf{Step 2: Polynomial projection for $p\ge \tau$.}
Suppose first that $p\ge\tau$. 
Applying that lemma to the phase-shifted cosine in $a_{\bm\xi}$ gives
\begin{equation}
\label{eq:explicit-BarronApprox-projection}
\begin{aligned}
\|(I- \mathcal{Q}_{p-1})a_{\bm\xi}\|_{L^2(U_N,\rho_N)}
\le\frac{\min\{1,(2\pi e\|\bm\xi\|_1/p)^p\}}
{1+\|\bm\xi\|_1^\tau}
\le\left(\frac{2\pi e}{p}\right)^\tau,
\end{aligned}
\end{equation}
where we have used $p!\ge(p/e)^p$ and $\min\{1,t^p\}\le t^\tau$ for $t\ge0$ and $p\ge\tau$.

For an integer $n\ge1$, let $\bm\xi_1,\ldots,\bm\xi_n$ be independent random variables with density
$\varpi$. 
Set $f_n:=\frac{B_e}{n}\sum_{j=1}^n a_{\bm\xi_j}$,
and
\eqref{eq:explicit-BarronApprox-representation_1} gives $\mathbb E_{\bm\xi_1,\cdots,\bm\xi_n} f_n = g $.
Independence makes the cross terms with distinct indices vanish, 
and thus, with \(Z=(I- \mathcal{Q}_{p-1})a_{\boldsymbol\xi}\), there holds 
\begin{equation*}
\begin{aligned}
\mathbb E_{\bm\xi_1,\cdots,\bm\xi_n} \|(I- \mathcal{Q}_{p-1})(g-f_n)\|_{L^2(U_N,\rho_N)}^2
&=\frac{B_e^2}{n}
\left(\mathbb E_{\bm\xi} \| Z \|_{L^2(U_N,\rho_N)}^2
-\| \mathbb E_{\bm\xi} Z \|_{L^2(U_N,\rho_N)}^2\right)\\
&\le\frac{B_e^2}{n}\left(\frac{2\pi e}{p}\right)^{2\tau},
\end{aligned}
\end{equation*}
where the last inequality follows from
\eqref{eq:explicit-BarronApprox-projection}.
Choose samples attaining this bound and set
$\mathfrak{f}:= \mathcal{Q}_{p-1}(g-f_n)\in\mathbb{P}_{p-1}$. Then $\mathfrak{f}$ is real valued and
\begin{equation}
\label{eq:explicit-BarronApprox-polynomial-error}
\|g-f_n-\mathfrak{f}\|_{L^2(U_N,\rho_N)}
=\|(I- \mathcal{Q}_{p-1})(g-f_n)\|_{L^2(U_N,\rho_N)}
\le B_e\left(\frac{2\pi e}{p}\right)^\tau n^{-1/2}.
\end{equation}
Set $n:=K-p^N$.
Since $K\ge2$, with the choice \eqref{p_KN}, 
we have $p^N \le \frac{K}{2}$ and $p \ge \frac14 K^{1/N}$.
It follows that
\[
n=K-p^N\ge\frac{K}{2},
\qquad
n^{-1/2}\le\sqrt2\,K^{-1/2},
\qquad
\left(\frac{2\pi e}{p}\right)^\tau
\le(8\pi e)^\tau K^{-\tau/N}.
\]
Thus, $p\ge\tau$ and \eqref{eq:explicit-BarronApprox-polynomial-error} give the explicit chain
\[
\begin{aligned}
\|g-f_n-\mathfrak{f}\|_{L^2(U_N,\rho_N)}
\le B_e\left(\frac{2\pi e}{p}\right)^\tau n^{-1/2}
\le\sqrt2\,B_e(8\pi e)^\tau
K^{-\frac12-\frac\tau N}.
\end{aligned}
\]

Now, it remains to approximate $\mathfrak{f}$ by a cosine neural network.
If $p=1$, then $\mathfrak{f}$ is constant and is represented exactly by one additional cosine unit with zero frequency and zero phase. 
If $p\ge2$, apply \cite[Theorem~3]{ScarselliTsoi1998} with
degree $p-1$: it gives uniform approximation on $U_N$ using at most
\[
s_p:=(p-1)\binom{p+N-2}{N-1}
\le (p-1)p^{N-1}\le p^N - 1
\]
hidden units.  
Then, the approximant has no separate output bias and uses at most \(s_p+1\le p^N\) hidden units.
Thus, in either case, we can choose $\widetilde{\mathfrak{f}}$ such that
\[
\|\mathfrak{f}-\widetilde{\mathfrak{f}}\|_{L^2(U_N,\rho_N)}
\le(2-\sqrt2)B_e(8\pi e)^\tau K^{-\frac12-\frac\tau N}.
\]
Then, the network $g_K:=f_n+\widetilde{\mathfrak{f}}$ has at most $n+p^N=K$ hidden units.
The triangle inequality and \eqref{eq:explicit-BarronApprox-extension} give
\begin{equation}
\label{eq:explicit-BarronApprox-large-width}
\begin{aligned}
\|g-g_K\|_{L^2(U_N,\rho_N)}
\le2B_e(8\pi e)^\tau K^{-\frac12-\frac\tau N}
\le4(8\pi e)^\tau
\|g\|_{\mathbb B_\tau(U_N)}K^{-\frac12-\frac\tau N}.
\end{aligned}
\end{equation}


\textbf{Step 3: Approximation for $p< \tau$.}
We can also use the same averaging argument without $ \mathcal{Q}_{p-1}$.
Since $\|a_{\bm\xi}\|_{L^2(U_N,\rho_N)}\le1$, it gives a choice of $n$
samples for which
\begin{equation}
\label{eq:explicit-BarronApprox-sampling}
\|g-f_n\|_{L^2(U_N,\rho_N)}\le B_e n^{-1/2}.
\end{equation}
If $p<\tau$, then $(K/2)^{1/N}<p+1<\tau+1$, and hence $K^{\tau/N}<2^{\tau/N}(\tau+1)^\tau \le[2(\tau+1)]^\tau$.
Taking $n=K$ in \eqref{eq:explicit-BarronApprox-sampling} and using
\eqref{eq:explicit-BarronApprox-extension}, we obtain a network $g_K$
with at most $K$ hidden units satisfying
\[
\|g-g_K\|_{L^2(U_N,\rho_N)}
\le B_e K^{-1/2}
\le2[2(\tau+1)]^\tau
\|g\|_{\mathbb B_\tau(U_N)}K^{-\frac12-\frac\tau N}.
\]
Combining these cases proves \eqref{eq:explicit-BarronApprox-error}.
\end{proof}


\section{Input truncation errors}


The following result quantifies the error arising from truncating the PDE coefficient.

\begin{lemma}
\label{lem:input-truncation}
Consider the elliptic equation \eqref{eq_case1_darcyflow} with
$f\in L^2(\Omega)$ under Assumption~\ref{ass:ellipticity}. 
Assume that $\Omega\subset\R^d$ is a bounded $C^{1,1}$ domain, such that,
for every $n\ge1$, $\by_n\in U_n$, and $h\in L^2(\Omega)$, the weak solution
$v\in H_0^1(\Omega)$ belongs to $H^2(\Omega)$ and satisfies the
uniform estimate
\begin{equation}
\label{eq:uniform-H2-regularity}
\|v\|_{H^2(\Omega)}
\le C_{\mathrm{reg}}\|h\|_{L^2(\Omega)},
\qquad
-\operatorname{div}(a_n(\cdot;\by_n)\nabla v)=h,
\end{equation}
where $C_{\mathrm{reg}}$ is independent of $n$, $\by_n$, and $h$.
For $j\ge1$, set $\bm\xi_j:=\bk(j)/L$ and
let $u_n(\by_n)$ and $u_N(\by_N)$ be the solutions corresponding to
$a_n(\cdot;\by_n)$ and $a_N(\cdot;\by_N)$, respectively. Then there is a
constant $C_{\mathrm{truc}}>0$, depending only on $\Omega$ and
$C_{\mathrm{reg}}$, such that, for every $N\ge1$,
\begin{equation}
\label{eq:input-truncation-oscillatory}
\sup_{n\ge N}\|u_n-u_N\|_{L^2(U_n,\rho_n;L^2(\Omega))}
\le C_{\mathrm{truc}}\|f\|_{L^2(\Omega)}
\sum_{j\ge N}\frac{|b_j|}{\|\bm\xi_j\|_2},
\end{equation}
where $u_N(\by_N)$ is regarded as a function on $U_n$ appended by zeros.

Suppose, in addition, that the enumeration in \eqref{eq_basis} is
chosen by nondecreasing frequency and that
$|b_j|\le C_{\sigma}(j+1)^{-r}$,
$j\in\N_0$,
for some $C_{\sigma}>0$ and $r>0$. Then there exist
$c_{d,L}>0$ and $N_{d,L}\ge1$, depending only on $d$, $L$, such that, if $r+1/d>1$,
then for every $N\ge N_{d,L}$, 
\begin{equation}
\label{eq:input-truncation-rate}
\sup_{n\ge N}\|u_n-u_N\|_{L^2(U_n,\rho_n;L^2(\Omega))}
\le
\frac{C_{\mathrm{truc}}C_{\sigma}}{c_{d,L}(r+1/d-1)}
\|f\|_{L^2(\Omega)}
N^{1-r-1/d}.
\end{equation}
\end{lemma}

\begin{proof}
Fix $n\ge N$ and $\by_n\in U_n$, set $\by_N:=\mathcal P_N\by_n$, and write
$w_{n,N}:=u_n(\by_n)-u_N(\by_N)$. Subtracting the variational equations
\eqref{eq:elliptic-pde_trunc} at levels $n$ and $N$ gives
\begin{equation}
\label{eq:input-truncation-equation}
\int_\Omega a_n(\cdot;\by_n)\nabla w_{n,N}\cdot\nabla\varphi\,\dd x
=-\int_\Omega(a_n(\cdot;\by_n)-a_N(\cdot;\by_N))
\nabla u_N\cdot\nabla\varphi\,\dd x,
~~ \varphi\in H_0^1(\Omega).
\end{equation}
Let $z_{n,N}\in H_0^1(\Omega)$ solve the adjoint problem
\begin{equation}
\label{eq:input-truncation-adjoint}
\int_\Omega a_n(\cdot;\by_n)\nabla z_{n,N}\cdot\nabla\varphi\,\dd x
=\int_\Omega w_{n,N}\varphi\,\dd x,
\qquad \varphi\in H_0^1(\Omega).
\end{equation}
Using $z_{n,N}$ in \eqref{eq:input-truncation-equation}, $w_{n,N}$ in
\eqref{eq:input-truncation-adjoint}, and
$a_n-a_N=\sum_{j=N}^{n-1}b_jy_j\psi_j$, we obtain
\begin{equation}
\label{eq:input-truncation-duality}
\|w_{n,N}\|_{L^2(\Omega)}^2
=-\sum_{j=N}^{n-1}b_jy_j
\int_\Omega\psi_j\nabla u_N\cdot\nabla z_{n,N}\,\dd x.
\end{equation}

For each mode in \eqref{eq_basis}, there is a function
$\widetilde\psi_j$ with
$\|\widetilde\psi_j\|_{L^\infty(\Omega)}\le1$ such that
\begin{equation}
\label{eq:trigonometric-antiderivative}
\begin{gathered}
\widetilde\psi_{2\ell-1}(x)
:=\sin\left(2\pi\bm\xi_{2\ell-1}\cdot x\right),
\qquad
\widetilde\psi_{2\ell}(x)
:=-\cos\left(2\pi\bm\xi_{2\ell}\cdot x\right),
\qquad \ell\ge1,\\
\psi_j
=\frac{1}{2\pi\|\bm\xi_j\|_2}
\partial_{\mathbf e_j}\widetilde\psi_j,
\qquad
\mathbf e_j:=\frac{\bm\xi_j}{\|\bm\xi_j\|_2}.
\end{gathered}
\end{equation}
Applying integration by parts in
\eqref{eq:trigonometric-antiderivative}, with
$h_{n,N}:=\nabla u_N\cdot\nabla z_{n,N}$, yields
\begin{equation}
\label{eq:oscillatory-integration-by-parts}
\int_\Omega\psi_jh_{n,N}\,\dd x
=\frac{1}{2\pi\|\bm\xi_j\|_2}
\left(
\int_{\partial\Omega}\widetilde\psi_jh_{n,N}
\mathbf e_j\cdot\mathbf n\,\dd s
-\int_\Omega\widetilde\psi_j
\partial_{\mathbf e_j}h_{n,N}\,\dd x
\right),
\end{equation}
where $\mathbf n$ is the exterior unit normal. The trace theorem~\cite[Section~5.5]{evans2010partial},
the product rule, and the Cauchy--Schwarz inequality give
\begin{equation}
\label{eq:oscillatory-integral-bound}
\left|\int_\Omega\psi_j
\nabla u_N\cdot\nabla z_{n,N}\,\dd x\right|
\le \frac{C_\Omega}{\|\bm\xi_j\|_2}
\|u_N\|_{H^2(\Omega)}\|z_{n,N}\|_{H^2(\Omega)}.
\end{equation}
Applying \eqref{eq:uniform-H2-regularity} at levels $N$ and $n$ gives
\[
\|u_N\|_{H^2(\Omega)}
\le C_{\mathrm{reg}}\|f\|_{L^2(\Omega)},
\qquad
\|z_{n,N}\|_{H^2(\Omega)}
\le C_{\mathrm{reg}}\|w_{n,N}\|_{L^2(\Omega)}.
\]
Substituting these estimates and
\eqref{eq:oscillatory-integral-bound} into
\eqref{eq:input-truncation-duality}, using $|y_j|\le1$, and
cancelling one factor $\|w_{n,N}\|_{L^2(\Omega)}$ gives
\[
\|w_{n,N}\|_{L^2(\Omega)}
\le C_{\mathrm{truc}}\|f\|_{L^2(\Omega)}
\sum_{j=N}^{n-1}\frac{|b_j|}{\|\bm\xi_j\|_2}.
\]
This bound is uniform in $\by_n$. Integrating it with respect to the
probability measure $\rho_n$ and taking the supremum over $n\ge N$ proves
\eqref{eq:input-truncation-oscillatory}.

It remains to show \eqref{eq:input-truncation-rate}. 
Noticing
$\#\{\bk\in\mathbb Z^d:\|\bk\|_2\le R\}
\le(2R+1)^d$
we clearly have
\begin{equation}
\label{eq:ordered-frequency-growth}
\|\bm\xi_j\|_2\ge c_{d,L}(j+1)^{1/d},
\qquad j\ge N_{d,L},
\end{equation}
for suitable $c_{d,L}>0$ and $N_{d,L}\ge1$.
Combining the assumed decay of $|b_j|$ with
\eqref{eq:ordered-frequency-growth}, and using the integral test,
we obtain
\[
\sum_{j\ge N}\frac{|b_j|}{\|\bm\xi_j\|_2}
\le\frac{C_{\sigma}}{c_{d,L}}
\sum_{j\ge N}(j+1)^{-r-1/d}
\le\frac{C_{\sigma}}{c_{d,L}(r+1/d-1)}
N^{1-r-1/d}.
\]
Together with \eqref{eq:input-truncation-oscillatory}, this proves
\eqref{eq:input-truncation-rate}.
\end{proof}